\documentclass[pdflatex,sn-mathphys-ay]{sn-jnl}

\usepackage{graphicx}%
\usepackage{multirow}%
\usepackage{amsmath,amssymb,amsfonts}%
\usepackage{amsthm}%
\usepackage{mathrsfs}%
\usepackage[title]{appendix}%
\usepackage{xcolor}%
\usepackage{textcomp}%
\usepackage{manyfoot}%
\usepackage{booktabs}%
\usepackage{algorithm}%
\usepackage{algorithmicx}%
\usepackage{algpseudocode}%
\usepackage{listings}%

\theoremstyle{thmstyleone}%
\newtheorem{theorem}{Theorem}
\usepackage{amsmath,amssymb,amsfonts}
\usepackage{mathtools}
\usepackage[scr=rsfs]{mathalpha}
\usepackage[title]{appendix}
\usepackage{xcolor}
\usepackage{enumitem}
\usepackage[T1]{fontenc}
\usepackage[utf8]{inputenc}

\definecolor{myred}{RGB}{255,50,50}
\definecolor{myblack}{RGB}{0,0,0}
\definecolor{mygreen}{RGB}{50,200,50}

\newcommand{\N}{\mathbb{N}}

\newcommand{\R}{\mathbb{R}}

\newcommand{\HH}{\R^n}

\newcommand{\inner}[2]{\langle #1,#2\rangle}
\newcommand{\Inner}[2]{\left\langle #1,#2\right\rangle}
\newcommand{\norm}[1]{\|{#1}\|}

\newcommand{\veps}{{\varepsilon}}
\newcommand{\what}[1]{\widehat #1}

\newcommand{\set}[1]{\{#1\}}

\newcommand{\seq}[1]{\left(#1\right)}

\newcommand{\lab}[1]{\label{#1}}
\newcommand{\bprop}{\begin{proposition}}
\newcommand{\eprop}{\end{proposition}}
\newcommand{\blemm}{\begin{lemma}}
\newcommand{\elemm}{\end{lemma}}
\newcommand{\bdefi}{\begin{definition}}
\newcommand{\edefi}{\end{definition}}
\newcommand{\btheo}{\begin{theorem}}
\newcommand{\etheo}{\end{theorem}}
\newcommand{\bproo}{\begin{proof}}
\newcommand{\eproo}{\end{proof}}
\newcommand{\brema}{\begin{remark}}
\newcommand{\erema}{\end{remark}}
\newcommand{\bitem}{\begin{itemize}}
\newcommand{\eitem}{\end{itemize}}
\newcommand{\bassu}{\begin{assumption}}
\newcommand{\eassu}{\end{assumption}}
\newcommand{\bcoro}{\begin{corollary}}
\newcommand{\ecoro}{\end{corollary}}
\newcommand{\benum}{\begin{enumerate}[label = \emph{(\alph*)}]}
\newcommand{\eenum}{\end{enumerate}}

\newcommand{\mgap}{\vspace{.1in}}

\DeclareMathOperator{\dom}{dom}
\DeclareMathOperator{\dist}{dist}
\DeclareMathOperator{\Gap}{Gap}
\DeclareMathOperator{\interior}{int}
\DeclareMathOperator*{\argmin}{argmin}

\DeclareMathOperator{\prox}{\mbox{prox}}

\newtheorem{assumption}{Assumption}
\newtheorem{lemma}[theorem]{Lemma}
\newtheorem{corollary}[theorem]{Corollary}
\newtheorem{proposition}[theorem]{Proposition}
\theoremstyle{thmstyletwo}%
\newtheorem{remark}{Remark}%

\theoremstyle{thmstylethree}%
\newtheorem{definition}{Definition}%

\begin{document}

\title{A Bregman inertial iteratively regularized extragradient method for bilevel variational inequality problems}








\author[1]{\fnm{Kangming} \sur{Chen}}
\email{kangming@tmu.ac.jp}

\author[2]{\fnm{Atsushi} \sur{Hori}}
\email{hori@nitech.ac.jp}

\author*[3]{\fnm{Ellen H.} \sur{Fukuda}}
\email{ellen@i.kyoto-u.ac.jp}

\affil[1]{
  \orgdiv{Faculty of Economics and Business Administration},
  \orgname{Tokyo Metropolitan University},
  \orgaddress{
    \street{1-1 Minami-Osawa},
    \city{Hachioji},
    \postcode{192--0397},
    \state{Tokyo},
    \country{Japan}
  }
}

\affil[2]{
  \orgdiv{Graduate School of Engineering},
  \orgname{Nagoya Institute of Technology},
  \orgaddress{
    \street{Gokiso-cho, Showa-ku},
    \city{Nagoya},
    \postcode{466--8555},
    \country{Japan}
  }
}

\affil*[3]{
  \orgdiv{Graduate School of Informatics},
  \orgname{Kyoto University},
  \orgaddress{
    \street{Yoshida-Honmachi},
    \city{Kyoto},
    \postcode{606--8501},
    \country{Japan}
  }
}


\abstract{In this paper, we consider bilevel variational inequality problems, where the feasible set is the solution set of another variational inequality. We propose a Bregman inertial regularized extragradient method for solving these problems. By performing the inertial extrapolation in the dual space, the proposed method aligns the inertial step with the structure of the Bregman three-point identity.
We also derive explicit non-asymptotic bounds for the gap function. More precisely, using diminishing regularization, we establish convergence rates of $\mathcal{O}(1/k^{1-b})$ for the optimality gap and $\mathcal{O}(1/k^b)$ for the feasibility gap with $ 0 < b< 1$. For constant regularization parameter $\eta>0$, we establish an $\mathcal{O}(1/k)$ rate for the optimality gap and an $\mathcal{O}(1/k)+\mathcal{O}(\eta)$ bound for the feasibility gap.
Furthermore, numerical experiments on traffic networks and multi-portfolio Nash equilibrium problems demonstrate the practical effectiveness of the proposed framework.}





\maketitle

\section{Introduction}

We consider the following (bilevel) variational inequality problem:
\begin{align} \lab{eq:prob}
\mbox{Find} \enspace x \in Q \enspace \mbox{such that} \enspace
\Inner{H(x)}{z - x} \geq 0 \enspace \text{for all} \enspace z \in Q, \tag{VIP($H$,$Q$)}
\end{align}
where $ Q $ is the solution set of the following variational inequality:
\begin{align} \lab{eq:def.Q}
 Q \coloneqq \set{x \in X \mid \Inner{F(x)}{y - x} \geq 0 \enspace \text{for all} \enspace y \in X}. \tag{VIP($F$,$X$)}
\end{align}
Problem~\eqref{eq:prob} arises naturally in equilibrium selection. For instance, consider a noncooperative game whose Nash equilibria are characterized as the solution set of a variational inequality \eqref{eq:def.Q}. When the game admits multiple equilibria, it is often desirable to select one satisfying an additional criterion, such as maximizing social welfare
; see, e.g., \cite[Subsection~1.1]{SamYou25}.
Such a secondary selection objective can be modeled through another operator~$H$.
This leads to the problem of finding a solution of the upper-level variational inequality~\eqref{eq:prob}. 
Therefore, bilevel variational inequalities provide a natural framework for equilibrium selection in game theory. 

Regarding methods, a classical approach to solving variational inequality problems is the extragradient method in~\cite{Kor76}. For problem \eqref{eq:prob}, the corresponding update is given by
\begin{align}
  \begin{split} 
    y_k & = P_Q(x_k - \lambda_k H(x_k)), \\ 
    x_{k+1} & = P_Q(x_k - \lambda_k H(y_k)),
  \end{split}
  \lab{eq:kor.intr}
\end{align}
where $\lambda_k > 0$ is the stepsize, and $P_Q(\cdot)$ denotes the orthogonal projection onto~$Q$. However, the hierarchical structure of \eqref{eq:def.Q} makes $P_Q(\cdot)$ difficult to compute since $Q$ itself is a solution set.
To address this issue, a regularization-based strategy is often employed, e.g., \cite{SamYou25};
the method approximates the bilevel structure using a regularized operator $F + \eta_k H$ over the feasible set $X$, whose projection $P_X(\cdot)$ is assumed to be easy to compute. The regularized extragradient update is given by
\begin{align}
\begin{split}
y_k
&= P_X\!\left(x_k-\lambda_k(F(x_k)+\eta_k H(x_k))\right), \\
x_{k+1}
&= P_X\!\left(x_k-\lambda_k(F(y_k)+\eta_k H(y_k))\right),
\end{split}
\label{eq:regkor.intr}
\end{align}
where $\eta_k$ is a regularization parameter.
Recently,~\cite{SamYou25} established iteration complexity bounds for \eqref{eq:regkor.intr} in terms of the dual gap function (Definition~\ref{def:dualgap}) under suitable conditions on~$\lambda_k$ and~$\eta_k$.
Further developments on hierarchical variational inequalities include regularization frameworks and extragradient methods with complexity guarantees (see, e.g.,~\cite{cortild2025regularization,dvurechensky2025extragradient}),
as well as stochastic variance-reduced extragradient methods~(\cite{dvurechensky2026stochastic}).

To accelerate convergence, inertial techniques have been incorporated into algorithms for monotone inclusions and variational inequalities since the seminal work of~\cite{AlvAtt01}.
In standard Euclidean settings, ~\cite{alves2025inertial} introduced an inertial effect in the current iterate $x_k$.
This yields the extrapolated base point:
\begin{align} \label{eq:primal_inertia}
    w_k^{\text{Euc}} = x_k + \alpha_k(x_k - x_{k-1}),
\end{align}
where $\alpha_k \ge 0$. The subsequent extragradient updates are then performed using $w_k^{\text{Euc}}$ instead of $x_k$.

Primal extrapolation like \eqref{eq:primal_inertia} is associated with Euclidean geometry and may not be aligned with the geometry induced by a general Bregman divergence. For this reason, we generalize the inertial regularized extragradient method by replacing the Euclidean projection with a Bregman proximal mapping and by constructing the inertial extrapolation in the corresponding dual space.
Note that Bregman divergence is generated by a strongly convex function $\omega$ that can be selected
according to both the geometry of the feasible set and the structure of the problem. 
Indeed, the Bregman three-point identity and the associated proximal analysis are expressed in terms of gradient differences $\nabla\omega(x)-\nabla\omega(y)$.

Thus, unlike standard approaches that perform linear extrapolation in the primal space as in~\eqref{eq:primal_inertia}, we construct the extrapolated point by aggregating the inertial extrapolation in the dual space induced by the distance-generating function \(\omega\). This idea is used in recent developments on dual-space acceleration and Bregman optimization; see, e.g.~\cite{chen2017accelerated,jolaoso2022analysis,krichene2015accelerated,nesterov2007dual,nesterov2009primal,wang2024mirror}.
To be more specific, we introduce a dual inertial extrapolation step that incorporates momentum through the gradient mapping \(\nabla\omega\), making it compatible with the underlying Bregman geometry.
The dual-extrapolated point $w_k$ is defined by
\begin{align} \label{eq:dual_inertia_intro}
    w_k = \nabla \omega^*((1+\alpha_k)\nabla\omega(x_k) - \alpha_k \nabla\omega(x_{k-1})),
\end{align}
where $\nabla \omega^*$ denotes the gradient of the convex conjugate function $\omega^*$ of $\omega$. 

Using $w_k$ as the proximal center, our proposed \textbf{Bregman Inertial Iteratively Regularized Extragradient}, abbreviated as \textbf{IneIREG-Bregman} method, performs the regularized extragradient updates as follows:
\begin{align}
    y_k &= \underset{u \in X}{\operatorname{argmin}} \left\{ \lambda_k \langle F(w'_k) + \eta_k H(w'_k), u \rangle + D_\omega(u, w_k) \right\}, \label{eq:y_bregman} \\
    x_{k+1} &= \underset{u \in X}{\operatorname{argmin}} \left\{ \lambda_k \langle F(y_k) + \eta_k H(y_k), u \rangle + D_\omega(u, w_k) \right\}, \label{eq:x_bregman}
\end{align}
where $w'_k \coloneqq \text{argmin}_{u \in \Omega} D_\omega(u, w_k)$ is the Bregman projection of $w_k$ onto the set $\Omega$. Here, $\Omega$ is an auxiliary set that contains $X$ and lies within the effective domains of $F$ and $H$, thus ensuring that $F(w'_k)$ and $H(w'_k)$ are well-defined.



Such flexibility allows the distance-generating function to be selected according to both the geometry of the feasible set and the structure of the problem. For instance, when $X$ is the probability simplex, the Kullback--Leibler (KL) divergence provides a natural alternative to the Euclidean distance, particularly when the problem structure is compatible with entropic geometry.
In particular, when the distance-generating function is chosen as the squared Euclidean norm (i.e., $\omega(x) = \frac{1}{2}\|x\|^2$), the gradient mapping becomes the identity function $\nabla \omega(x) = x$.
Thus, the dual extrapolation step in \eqref{eq:dual_inertia_intro} reduces to the standard primal linear extrapolation \eqref{eq:primal_inertia}. In this case,  the \textbf{IneIREG-Bregman} algorithm includes the Euclidean \textbf{IneIREG} method proposed by~\cite{alves2025inertial} as a special case.

\mgap

\noindent
{\bf Main contributions.} 
Our main results and contributions can be summarized as follows:

\begin{itemize}
    \item \textbf{Dual Inertial Extrapolation Framework.} 
    We develop a Bregman-regularized extragradient method for \eqref{eq:prob} using dual momentum to preserve feasible geometry. 
    \item \textbf{Explicit Complexity Bounds.} 
    We give an iteration complexity analysis  using a dual gap function. Under diminishing regularization parameters $\eta_k = \eta_0 / (k+1)^b$ ($b\in(0,1)$), we establish an $\mathcal{O}(1/k^{1-b})$ bound for the optimality gap and an $\mathcal{O}(1/k^b)$ bound for the feasibility gap. Furthermore, under constant regularized parameters $\eta_k\equiv \eta$, we establish an $\mathcal{O}(1/k)$ bound for the optimality gap and an $\mathcal{O}(1/k) + \mathcal{O}(\eta)$ bound for the feasibility gap.

    \item \textbf{Computational Efficiency.} 
    Using a Bregman geometry adapted to the feasible set, the proposed method admits explicit updates and avoids the Euclidean projections required by existing methods. In particular, numerical experiments on traffic network and Nash equilibrium problems illustrate the computational advantages of the proposed method. 
\end{itemize}

\noindent
{\bf Organization of the paper.} 
Section~\ref{sec:pre} introduces the general notation and some preliminaries. 
In Section~\ref{sec:main}, we give our main algorithm, \textbf{IneIREG-Bregman}, together with the basic assumptions on~\eqref{eq:prob}.
Section~\ref{sec:hardy} presents iteration complexity results for our proposed algorithm.
Specialized results for the cases of diminishing and constant regularization parameter sequences are then discussed in Section~\ref{subsec:dimish}. 
Finally, in Section~\ref{sec:ne}, we report some numerical experiments and conclude the paper in Section~\ref{sec:conclusions}.

\section{General notation and preliminaries}
\label{sec:pre}

Let $\HH$ be the $n$-dimensional Euclidean space with the inner product $\inner{\cdot}{\cdot}$ and the induced norm $\norm{\cdot}$. For a matrix $Z$, we use $Z^\top$ to denote its transpose. For any nonempty closed convex set $Y \subseteq \HH$, $P_Y(x)$ denotes the orthogonal projection of $x \in \HH$ onto $Y$.
The distance from the point $x\in\HH$ to the set $Y\subseteq\HH$ is given by $\dist(x, Y) \coloneqq \inf_{y \in Y} \|x - y\|$.
We say that $Y$ is \emph{simple} if $P_Y(\cdot)$ is computationally inexpensive. An operator $F \colon Y \subseteq \HH \to \HH$ is said to be \emph{monotone} on $Y$ if $\inner{F(x) - F(y)}{x - y} \geq 0$ for all $x, y \in Y$, and $L_F$-\emph{Lipschitz continuous} if there exists $L_F>0$ such that $\norm{F(x) - F(y)} \leq L_F \norm{x - y}$ for all $x, y\in Y$.

We denote  the \emph{diameter} and \emph{Bregman diameter} of the compact and convex set $ Y \subset \HH $ as
$D_Y \coloneqq \sup_{x,y\in Y}\,\norm{x - y}$ and
$\Omega_Y^2 \coloneqq \sup_{x,y\in Y}\,D_{\omega}(x, y)$, respectively, where the definition of the Bregman divergence $D_\omega$ is given in Definition~\ref{def:Bregman.divergence}.
Let $\R_+$ be the set of nonnegative real numbers.
For a set $A$, its interior is written as $\interior(A)$. Also, for an extended-valued function $\psi \colon \R^n \to \R \cup \{+\infty\}$, its effective domain is defined by $\dom \psi \coloneqq \{ x \in \HH \mid \psi(x) < +\infty \}$. Moreover, the conjugate of $\psi$ is given by 
$\psi^*(y) \coloneqq \sup_{x \in \dom \psi} (\inner{y}{x} - \psi(x))$.
\begin{assumption}
\label{assu:Hbounded}
The operator $H$ is bounded on both $X$ and $Q$. Equivalently,
\begin{align}\lab{eq:def.ch}
C_H\coloneqq \sup_{x\in X}\,\norm{H(x)} <+\infty\quad \mbox{and} \quad B_H \coloneqq \sup_{x\in Q}\,\norm{H(x)}<+\infty.
\end{align}
\end{assumption}

We also define the following function, which is equivalent to a merit function for VI problems.
\bdefi[Dual gap function] \lab{def:dualgap}\sloppy
The \emph{dual gap function} $\text{Gap}(\cdot, H, Q) \colon \HH \to \R \cup\{+\infty\}$ associated with $\operatorname{VIP}(H, Q)$ is defined~as
\[
\text{Gap}(z, H, Q) \coloneqq \sup_{x\in Q}\,\Inner{H(x)}{z - x}, \quad  z \in \HH.
\]
In the same way, we also define the \emph{dual gap function} $\text{Gap}(\cdot, F, X)$ 
associated with $\operatorname{VIP}(F, X)$.
\edefi
Note that $\Gap(z,H,Q)\ge 0$ for all $z\in Q$, and if $H$ is monotone, then
$\Gap(z, H, Q) = 0$ if and only if $z\in Q$ is a solution to $\operatorname{VIP}(H,Q)$~\cite[Proposition~2.3.15]{Facchinei2004}.
The following result gives a more general lower bound for this gap function.
Its proof follows from~\cite[Theorem 2]{SamYou25} (see also~\cite[Proposition 1.2]{alves2025inertial}). 

\begin{proposition}[{\cite[Theorem~2]{SamYou25}}] \lab{pro:optm02}
For all $ y \in \HH $, we have
\begin{align} \lab{eq:bh}
 \Gap(y, H, Q) \geq -B_H \dist(y, Q).
\end{align}
Moreover, if $ Q $ is $ \sigma $-weakly sharp
 of order $ \mathcal{M} \geq 1 $, 
that is, there exist $\sigma > 0$ and $\mathcal{M} \geq 1$ such that $\langle F(x), y - x \rangle \geq \sigma \big(\operatorname{dist}(y, Q)\big)^\mathcal{M}$ for all $x \in Q$ and $y \in X$,
 then 
\begin{align} \lab{eq:bh02}
0\leq \dist(y, Q) \leq \left( \dfrac{\Gap(y, F, X)}{\sigma}\right)^{\frac{1}{\mathcal{M}}}
\quad \mbox{for all} \quad y \in X.
\end{align}
\end{proposition}

Now we recall the definition of the Bregman distance.
\begin{definition}\label{def:Bregman.divergence}
The \emph{Bregman distance} (or \emph{divergence}) associated with the distance-generating function $\omega \colon \R^n \to \R \cup\{+\infty\}$ is a mapping $D_{\omega} \colon \dom \, \omega \times \interior (\dom \, \omega) \to \mathbb{R}$ defined by
\[
  D_{\omega}(x,y) \coloneqq \omega(x) - \omega(y) - \langle \nabla \omega(y), x-y \rangle.
\]
\end{definition}


Throughout this paper, we assume the following critical regularities on the distance-generating function $\omega$.
\begin{assumption}\label{assump:omega}
The function $\omega \colon \R^n \to \R \cup\{+\infty\}$ satisfies the following properties:
\begin{enumerate}[label=(\alph*)]
    \item \label{assump strong convex}Strong convexity: $\omega$ is $\mu$-strongly convex with respect to the norm $\|\cdot\|$.
    \item Legendre: $\omega$ is a function of Legendre type, i.e., it is proper, closed, essentially smooth\footnote{A function $\omega$ is essentially smooth if it is differentiable on $\interior (\dom \, \omega) \neq \emptyset$, and $\norm{\nabla \omega(x^k)} \to +\infty$ for every $\{ x^k \} \subset \interior (\dom \, \omega)$ that converges to a boundary point of $\dom \, \omega$.} and strictly convex (see, e.g.,~\cite{rockafellar1970convex}). 
\end{enumerate}
\end{assumption}
Note that these assumptions give two properties that are essential for our analysis. First, the $\mu$-strong convexity of $\omega$ implies the coercivity bound $D_\omega(x, y) \ge \frac{\mu}{2}\|x-y\|^2$ for all $x\in \operatorname{dom} \omega$ and $y\in \interior (\dom \omega)$. Second, the smoothness of the Legendre function restricts our dual-extrapolated and proximal iterations to the relative interior of the feasible set. Moreover,
in this case, $\interior (\dom \, \omega)$ coincides with the set of points where $\omega$ is differentiable.

The convergence analysis of Bregman-based proximal methods depends on the following classical three-point identity.

\begin{lemma}[{Three-point identity~\cite[Lemma~3.1]{chen1993convergence}}]\label{3-pointA}
    Let $a, b \in \interior (\dom \, \omega)$ and $c \in \operatorname{dom} \, \omega$. Then, the following equality holds:
    $$\langle\nabla \omega(b)-\nabla \omega(a), c-a\rangle = D_{\omega}(c, a) + D_{\omega}(a, b) - D_{\omega}(c, b).$$
\end{lemma}

\begin{lemma}[{\cite[Lemma 3.2]{chen1993convergence}}]\label{3-pointB} \sloppy
	For any proper closed convex function $\theta \colon \mathbb{R}^{n} \rightarrow \mathbb{R} \cup \{ +\infty \}$ and any $z \in \interior (\dom \, \omega)$, if $\omega$ is differentiable at $z_{+}\coloneqq\underset{x \in \dom \, \omega}{\operatorname{argmin}}\{\theta(x)+D_{\omega}(x, z)\}$, then 
	$$\theta(x)+D_{\omega}(x, z) \geq \theta\left(z_{+}\right)+D_{\omega}\left(z_{+}, z\right)+D_{\omega}\left(x, z_{+}\right) \quad \mbox{for all } x \in \operatorname{dom}\, \omega.$$
\end{lemma}
The three-point identity is a standard tool for analyzing Bregman proximal methods. For a linear function $\theta(x) \coloneqq \langle\xi, x\rangle$, this property recovers important relations of the mirror-prox framework (see, e.g.,  \cite[Lemma 1]{chen2017accelerated}). It characterizes the distances between the update $z_+$, the base point $z$, and a reference point $x$, which will play an important role for our convergence analysis.

\section{The main algorithm } \lab{sec:main}

Recall the bilevel variational inequality problem~\eqref{eq:prob}, 
where $Q$ is the solution set of \eqref{eq:def.Q}.
We assume the following conditions:

\bassu \lab{assu:erdos}
\begin{enumerate}[label=(\alph*)]
\item \lab{assu:erdos.seg}
$ F \colon \dom F \subseteq \HH \to \HH $ and $ H \colon \dom H \subseteq \HH \to \HH $ are monotone and Lipschitz continuous mappings with constants $ L_F > 0 $ and $ L_H > 0 $, respectively. 
\item \lab{assu:erdos.ter} 
$ X $ and $ \Omega $ are simple nonempty closed and convex subsets of $ \HH $ such that 
$ X \subset \Omega \subseteq \dom \omega \subseteq \dom F \cap \dom H$.
\item \lab{assu:erdos.qua}
The solution set $ Q $ of $ \text{VIP}(F, X) $ is nonempty.
\end{enumerate}
\eassu

To set the stage for our  framework, we first recall the inertial iteratively regularized extragradient (IneIREG) method recently introduced by~\cite{alves2025inertial} for solving \eqref{eq:prob}.

\begin{algorithm}[ht]
\caption{IneIREG (\cite{alves2025inertial})}
\label{alg:main}
\begin{algorithmic}[1]

\Require Initial guess $x_0=x_{-1}\in X$.

\For{$k=0,1,\dots$}

    \State Choose an inertial parameter $\alpha_k>0$, and set
    \begin{align}
    \label{eq:def.wk}
        w_k \coloneqq x_k+\alpha_k(x_k-x_{k-1}).
    \end{align}

    \State Choose a stepsize $\lambda_k>0$ and a regularization parameter
    $\eta_k>0$.

    \State Set
    \[
        w'_k \coloneqq P_\Omega(w_k).
    \]

    \State Compute
    \begin{align}
    \label{eq:yx}
    \begin{aligned}
        y_k
        &=
        P_X\Bigl(
        w_k-\lambda_k
        \bigl(
        F(w'_k)+\eta_k H(w'_k)
        \bigr)
        \Bigr),
        \\[2mm]
        x_{k+1}
        &=
        P_X\Bigl(
        w_k-\lambda_k
        \bigl(
        F(y_k)+\eta_k H(y_k)
        \bigr)
        \Bigr).
    \end{aligned}
    \end{align}

\EndFor

\end{algorithmic}
\end{algorithm}
While Algorithm~\ref{alg:main} has good theoretical properties, computing the Euclidean projection $P_X$ can be difficult depending on the structure of~$X$. To overcome this difficulty, we replace the Euclidean projection with a Bregman proximal mapping.
Using this mapping together with the inertial extrapolation step~\eqref{eq:dual_inertia_intro}, we obtain the proposed algorithm shown in Algorithm~\ref{alg:bregman}. As usual, we assume that the associated Bregman proximal mapping is computationally tractable.

\begin{algorithm}[htbp]
\caption{IneIREG-Bregman}
\label{alg:bregman}
\begin{algorithmic}[1]

\Require Initial guess
$x_0=x_{-1}\in X\cap\operatorname{int}(\operatorname{dom}\omega)$
and a distance-generating function $\omega$ with its conjugate $\omega^*$.

\For{$k=0,1,\dots$}

    \State Choose an inertial parameter $\alpha_k\ge 0$.

    \State Compute the extrapolated point in the dual space and map it back:
    \begin{equation}
    \label{eq:dual_inertial_step}
        w_k
        \coloneqq
        \nabla\omega^*
        \bigl(
        (1+\alpha_k)\nabla\omega(x_k)
        -
        \alpha_k\nabla\omega(x_{k-1})
        \bigr).
    \end{equation}

    \State Choose a stepsize $\lambda_k>0$ and a regularization parameter
    $\eta_k>0$.

    \State Set $w'_k$ via the Bregman projection:
    \begin{equation}
    \label{eq:bregman_proj}
        w'_k
        \coloneqq
        \operatorname*{argmin}_{u\in\Omega}
        D_\omega(u,w_k).
    \end{equation}

    \State Compute the intermediate and next iterates:
    \begin{align}
        y_k
        &=
        \operatorname{prox}^{\omega}_{
        \lambda_k(F(w'_k)+\eta_k H(w'_k))}
        (w_k),
        \label{eq:yxxk1}
        \\
        x_{k+1}
        &=
        \operatorname{prox}^{\omega}_{
        \lambda_k(F(y_k)+\eta_k H(y_k))}
        (w_k).
        \label{eq:yxxk2}
    \end{align}

\EndFor

\end{algorithmic}
\end{algorithm}

\mgap

Here, the Bregman proximal operator, denoted as $\operatorname{prox}^{\omega}_{\xi}(z)$ is defined as the unique solution to the following minimization problem:
\begin{equation}\label{pro proxi2}
    \operatorname{prox}^{\omega}_{\xi}(z) \coloneqq \argmin_{w \in X} \left\{ \langle \xi, w \rangle + D_{\omega}(w, z) \right\}.
\end{equation}
For instance, utilizing this operator, the intermediate update $y_k$ in Algorithm~\ref{alg:bregman} can be written as
\begin{equation*}
    y_k = \argmin_{w \in X} \left\{ \langle \lambda_k (F(w'_k) + \eta_k H(w'_k)), w \rangle + D_{\omega}(w, w_k) \right\}.
\end{equation*}

    Since $\omega$ is globally $\mu$-strongly convex from Assumption~\ref{assump:omega}\ref{assump strong convex}, $\omega$ is supercoercive (1-coercive), and hence $\operatorname{dom}\omega^*=\mathbb R^n$ by~\cite[Proposition~2.16]{bauschke_borwein1997}.
    Thus, for every $k\ge0$,
    \[
    (1+\alpha_k)\nabla\omega(x_k)
    -\alpha_k\nabla\omega(x_{k-1})
    \in
    \operatorname{int}(\operatorname{dom}\omega^*).
    \]
    Since $\omega$ is Legendre, $\nabla \omega \colon \mbox{int}(\mbox{dom } \omega) \to \mbox{int}(\mbox{dom } \omega^*)$ and
    $(\nabla\omega)^{-1}=\nabla\omega^*$.
    Therefore,
    \[
    w_k
    =
    \nabla\omega^*\!\left(
    (1+\alpha_k)\nabla\omega(x_k)
    -\alpha_k\nabla\omega(x_{k-1})
    \right)
    \in
    \operatorname{int}(\operatorname{dom}\omega),
    \]
    and hence $D_\omega(\cdot,w_k)$ is well-defined.

The projection step in \eqref{eq:bregman_proj} serves as a safeguard ensuring that the regularized operator evaluation $F(w'_k)+\eta_k H(w'_k)$ is well-defined. Indeed, $w_k$ is obtained via the map $\nabla\omega^*$ in \eqref{eq:dual_inertial_step}, and $\Omega \subseteq \dom \omega \subseteq \dom F \cap \dom H$ from Assumption~\ref{assu:erdos}~\ref{assu:erdos.ter}. From the same assumption, $\Omega$ is  a simple set, so the additional projection incurs only minor computational cost. Moreover,~\cite[Theorem~3.12(iii)]{bauschke_borwein1997} ensures that $w_k'=\argmin_{u\in\Omega}D_\omega(u,w_k)$ exists uniquely and belongs to $\Omega\cap\operatorname{int}(\operatorname{dom}\omega)$.

Now, to ensure that all the Bregman proximal steps in Algorithm~\ref{alg:bregman} are well-defined, we give the following lemma,
which generalizes~\cite[Theorem~3.12]{bauschke_borwein1997}.
\begin{lemma}
\label{lem:prox_well_defined}
Let $C\subseteq\mathbb R^n$ be a nonempty closed convex set satisfying
$C\cap\operatorname{int}(\operatorname{dom}\omega)\neq\emptyset$.
Suppose that $\omega$ satisfies \emph{Assumption~\ref{assump:omega}}.
Then, for every
$z\in\operatorname{int}(\operatorname{dom}\omega)$
and $\xi\in\mathbb R^n$, the problem
\begin{align}\label{prob:prox.comput}
\min_{u\in C\cap\operatorname{dom}\omega}
\left\{
\langle\xi,u\rangle+D_\omega(u,z)
\right\}
\end{align}
admits a unique minimizer $\bar u$, with
$\bar u\in C\cap\operatorname{int}(\operatorname{dom}\omega)$.
\end{lemma}

\begin{proof}
Since $\omega$ is $\mu$-strongly convex, namely,
$
D_\omega(u,z)\ge\frac{\mu}{2}\|u-z\|^2,
$
$u\mapsto \langle\xi,u\rangle+D_\omega(u,z)$
is coercive and lower semicontinuous in the closed set $C$;
hence, there is a minimizer, and its uniqueness follows from strong convexity.

It remains to prove that $\bar{u} \in \operatorname{int}(\operatorname{dom}\omega)$.
Suppose to the contrary that the minimizer
$\bar u$ of~\eqref{prob:prox.comput} satisfies
$\bar u\in C\cap(\operatorname{dom}\omega\setminus
\operatorname{int}(\operatorname{dom}\omega))$.
Choose
$u^\circ\in C\cap\operatorname{int}(\operatorname{dom}\omega)$
and set
$$
u_t=(1-t)\bar u+t u^\circ,\qquad t\in(0,1].
$$
Then, we have $u_t\in C\cap\operatorname{int}(\operatorname{dom}\omega)$.
By essential smoothness~\cite[Lemma~26.2]{rockafellar1970convex}, 
$$
\left\langle\nabla\omega(u_t),u^\circ-\bar u\right\rangle
\to-\infty
$$
when $t\downarrow0$. Therefore, by defining
$\phi(t) \coloneqq \langle\xi,u_t\rangle+D_\omega(u_t,z)$,
we have 
$$\phi'(t)
=
\left\langle
\xi+\nabla\omega(u_t)-\nabla\omega(z),
u^\circ-\bar u
\right\rangle
\to-\infty$$
when $t\downarrow0$. This shows that
$\phi(t)<\phi(0)$ for all sufficiently small $t>0$, contradicting the optimality of $\bar u$.
Thus, we have
$\bar u\in C\cap\operatorname{int}(\operatorname{dom}\omega)$.
\end{proof}

These observations yield the well-definedness of Algorithm~\ref{alg:bregman}.
\begin{proposition} 
\label{prop:algorithm_well_defined}
Suppose that \emph{Assumptions~\ref{assump:omega} and~\ref{assu:erdos}}
hold.
Then, \emph{Algorithm~\ref{alg:bregman}} is well-defined for every $k\ge0$,
i.e.,
$$
w_k\in\operatorname{int}(\operatorname{dom}\omega),
\quad
w'_k\in
\Omega\cap\operatorname{int}(\operatorname{dom}\omega),
\quad \mbox{and} \quad
y_k,x_{k+1}
\in
X\cap\operatorname{int}(\operatorname{dom}\omega).
$$
\end{proposition}

\begin{proof}
  It holds from Lemma~\ref{lem:prox_well_defined} and the discussions before it. 
\end{proof}

\section{Convergence analysis}\label{sec:hardy}

In this section, we present the global convergence and iteration complexity analysis of the proposed \textbf{IneIREG-Bregman} algorithm. In Section~\ref{subsec:general_convergence}, we first establish a single step descent inequality and estimate the error caused by the dual inertial extrapolation. Using these results, we derive non-asymptotic convergence rates for two settings in Section~\ref{subsec:dimish}: diminishing and constant regularization $\eta_k$.
In these cases, we establish a worst-case global iteration complexity of $\mathcal{O}(1/\epsilon^2)$ to achieve an $\epsilon$-approximate solution.

\subsection{General convergence properties}
\label{subsec:general_convergence}
We first derive a generalized three-point inequality specific to the regularized proximal updates of our method.

\begin{proposition}\label{prop 3point}
Let the sequences $(w_k)$, $(w'_k)$, $(y_k)$, and $(x_{k+1})$ be generated by \emph{Algorithm~\ref{alg:bregman}}. Moreover, let $G_k(\cdot) \coloneqq F(\cdot) + \eta_k H(\cdot)$ and its Lipschitz constant be $L_k \coloneqq L_F + \eta_k L_H$. Then, for all $x \in X$ and $k \ge 0$, the following inequality holds:
\begin{equation}\label{ieq:three.point}
    \lambda_k \langle G_k(y_k), y_k - x \rangle \le  D_\omega(x, w_k)  - D_\omega(x, x_{k+1}) - \left(1 - \frac{\lambda_k^2 L_k^2}{\mu^2}\right)D_\omega(y_k, w_k).
\end{equation}
\end{proposition}

\begin{proof}
Let $z_+ \coloneqq \text{argmin}_{u \in X} \{ \langle \xi, u \rangle + D_\omega(u, z) \}$.
It follows from Lemma~\ref{3-pointB} that for all $x \in \operatorname{dom} \omega$,
\begin{equation}\label{eq:3point}
    \langle \xi, z_+ - x \rangle \le D_{\omega}(x, z) - D_{\omega}\left(z_{+}, z \right)-D_{\omega}\left(x, z_{+}\right).
\end{equation}

Now we decompose the term $\lambda_k \langle G_k(y_k), y_k - x \rangle$ as follows:
\begin{equation}\label{eq:gap.ip.decomp}
    \lambda_k \langle G_k(y_k), y_k - x \rangle = \lambda_k \langle G_k(y_k), y_k - x_{k+1} \rangle + \lambda_k \langle G_k(y_k), x_{k+1} - x \rangle,
\end{equation}
and it suffices to show the upper bound of the right-hand side of the equality.

For the second term of the right-hand side of~\eqref{eq:gap.ip.decomp},
applying~\eqref{eq:3point} to $z=w_k$, $\xi = \lambda_k G_k(y_k)$ and $z_+=x_{k+1}$, we obtain
\begin{align}\label{ieq:gap.ip.step1}
    \lambda_k \langle G_k(y_k), x_{k+1} - x \rangle \le D_\omega(x, w_k) - D_\omega( x_{k+1}, w_k) - D_\omega(x, x_{k+1}).
\end{align}

For the first term of the right-hand side of~\eqref{eq:gap.ip.decomp}, this term is rewritten as follows:
\[
    \lambda_k \langle G_k(y_k), y_k - x_{k+1} \rangle = \lambda_k \langle G_k(w'_k), y_k - x_{k+1} \rangle + \lambda_k \langle G_k(y_k) - G_k(w'_k), y_k - x_{k+1} \rangle.
\]
Similarly to~\eqref{ieq:gap.ip.step1}, the first term of the right-hand side of the above equation is bounded as follows by using~\eqref{eq:3point} with $z=w_k$, $\xi=\lambda_k G_k(w'_k)$, $z_+=y_k$, and $x=x_{k+1}$:
\[
    \lambda_k \langle G_k(w'_k), y_k - x_{k+1} \rangle \le D_\omega( x_{k+1}, w_k) - D_\omega( y_k, w_k) - D_\omega(x_{k+1}, y_k).
\]
Hence, we obtain
\begin{align}\label{ieq:gap.ip.step2}
\begin{aligned}
\lambda_k\langle G_k(y_k),y_k-x_{k+1}\rangle\le & D_\omega(x_{k+1},w_k)-D_\omega(y_k,w_k)-D_\omega(x_{k+1},y_k)\\
& +\lambda_k\langle G_k(y_k) - G_k(w'_k), y_k - x_{k+1} \rangle.
\end{aligned}
\end{align}

To bound the inner product term in the right-hand side of~\eqref{ieq:gap.ip.step2},
we successively apply the Cauchy--Schwarz inequality, the $L_k$-Lipschitz continuity of~$G_k$,
and Young's inequality (i.e., $ab\le a^2/(2\mu) + \mu b^2/2$) as follows:
\begin{align*}
    \lambda_k \langle G_k(y_k) - G_k(w'_k), y_k - x_{k+1} \rangle &\le \lambda_k \|G_k(y_k) - G_k(w'_k)\| \, \|y_k - x_{k+1}\| \\
    &\le \lambda_k L_k \|y_k - w'_k\| \|y_k - x_{k+1}\| \\
    &\le \frac{\lambda_k^2 L_k^2}{2\mu} \|y_k - w'_k\|^2 + \frac{\mu}{2}\|y_k - x_{k+1}\|^2\\
    & \le \frac{\lambda_k^2 L_k^2}{\mu^2} D_\omega(y_k,w'_k) + D_\omega(x_{k+1},y_k),
\end{align*}
where the last inequality follows from the $\mu$-strong convexity of $\omega$, as well as $D_\omega(u,v)\ge \mu/2\|u-v\|^2$.
Now, recalling~\eqref{eq:bregman_proj}, since $w'_k = \text{argmin}_{u \in \Omega} D_\omega(u, w_k)$ is the Bregman projection of $w_k$ onto $\Omega$, and $y_k \in X \subset \Omega$, the generalized Pythagorean theorem for Bregman divergences asserts that $D_\omega(y_k, w'_k) + D_\omega(w'_k, w_k) \le D_\omega(y_k, w_k)$ (see \cite[Proposition 3.16]{bauschke_borwein1997}).
Since the Bregman divergence is nonnegative, we immediately obtain $D_\omega(y_k, w'_k) \le D_\omega(y_k, w_k)$. Substituting this into the above inequality, we get
\[
    \lambda_k \langle G_k(y_k) - G_k(w'_k), y_k - x_{k+1} \rangle \le \frac{\lambda_k^2 L_k^2}{\mu^2} D_\omega( y_k, w_k) + D_\omega(x_{k+1}, y_k).
\]
Thus,~\eqref{ieq:gap.ip.step2} is bounded as follows:
\begin{align}
& \, \lambda_k\langle G_k(y_k),y_k-x_{k+1}\rangle \nonumber \\ 
\le & \, D_\omega(x_{k+1},w_k)-D_\omega(y_k,w_k)-D_\omega(x_{k+1},y_k)\notag \\
& + \frac{\lambda_k^2 L_k^2}{\mu^2} D_\omega(y_k,w'_k) + D_\omega(x_{k+1},y_k)\notag \\
\le& \, D_\omega( x_{k+1}, w_k) - D_\omega( y_k, w_k) + \frac{\lambda_k^2 L_k^2}{\mu^2} D_\omega( y_k, w_k).\label{ieq:gap.ip.step3}
\end{align}

Combining~\eqref{ieq:gap.ip.step1} and~\eqref{ieq:gap.ip.step3} with~\eqref{eq:gap.ip.decomp} leads to
\begin{align*}
    \lambda_k \langle G_k(y_k), y_k - x \rangle &\le \left[ D_\omega( x_{k+1}, w_k) - D_\omega( y_k, w_k) + \frac{\lambda_k^2 L_k^2}{\mu^2} D_\omega( y_k, w_k) \right] \\
    &\quad + D_\omega(x, w_k) - D_\omega( x_{k+1}, w_k) - D_\omega(x, x_{k+1}).
\end{align*}
Rearranging the remaining terms yields the result. 
\end{proof}

\begin{lemma}[Dual Bregman Inertial Lemma]
\label{lem:dual_bregman_inertia}
Let $x_k$ and $x_{k-1}$ be any two points in $\interior(\dom\omega)$, and let the extrapolated point $w_k$ be defined via dual extrapolation:
\[ \nabla\omega(w_k) \coloneqq (1+\alpha_k)\nabla\omega(x_k) - \alpha_k\nabla\omega(x_{k-1}) \]
for some $\alpha_k \ge 0$.
Then, for any point $x \in \dom\omega$, the following inequality holds:
\begin{equation} \label{eq:dual_inertial_lemma}
    D_\omega(x, w_k) \le (1+\alpha_k)D_\omega(x, x_k) - \alpha_k D_\omega(x, x_{k-1}) + \delta_k,
\end{equation}
where
\begin{equation}
  \label{eq:delta}
  \delta_k \coloneqq \frac{\alpha_k(1+\alpha_k)}{2\mu} \|\nabla\omega(x_k) - \nabla\omega(x_{k-1})\|^2.
\end{equation}
\end{lemma}

\begin{proof}
Define $\Phi \coloneqq (1+\alpha_k)D_\omega(x, x_k) - \alpha_k D_\omega(x, x_{k-1}) - D_\omega(x, w_k)$.
We show that $\Phi \ge -\delta_k$.
By the definition of the Bregman divergence, we have
\begin{align*}
\Phi &= (1+\alpha_k) \big[ \omega(x) - \omega(x_k) - \langle \nabla\omega(x_k), x - x_k \rangle \big] \\
&\quad - \alpha_k \big[ \omega(x) - \omega(x_{k-1}) - \langle \nabla\omega(x_{k-1}), x - x_{k-1} \rangle \big] \\
&\quad - \big[ \omega(x) - \omega(w_k) - \langle \nabla\omega(w_k), x - w_k \rangle \big].
\end{align*}
Grouping the inner products involving $x$, we obtain $\langle -(1+\alpha_k)\nabla\omega(x_k) + \alpha_k\nabla\omega(x_{k-1}) + \nabla\omega(w_k), x \rangle$, which is identically zero by the definition of the dual extrapolation step;
thus, $\Phi$ is simplified as
\begin{align*}
\Phi &= -(1+\alpha_k)\big[ \omega(x_k) - \langle \nabla\omega(x_k), x_k \rangle \big] \\
&\quad + \alpha_k\big[ \omega(x_{k-1}) - \langle \nabla\omega(x_{k-1}), x_{k-1} \rangle \big] + \big[ \omega(w_k) - \langle \nabla\omega(w_k), w_k \rangle \big].
\end{align*}
By the property of the conjugate $\omega^*$, we get $\omega(v) - \langle \nabla\omega(v), v \rangle = -\omega^*(\nabla\omega(v))$ (see e.g.~\cite[Theorem~23.5]{rockafellar1970convex}).
Letting $z_k \coloneqq \nabla\omega(x_k)$ and $\hat{z}_k \coloneqq \nabla\omega(w_k)$, $\Phi$ is rewritten in the dual space as follows:
\begin{equation} \label{eq:phi_dual}
\Phi = (1+\alpha_k)\omega^*(z_k) - \alpha_k\omega^*(z_{k-1}) - \omega^*(\hat{z}_k).
\end{equation}
Since $\omega$ is $\mu$-strongly convex with respect to $\|\cdot\|$, its conjugate $\omega^*$ is $(1/\mu)$-smooth (see e.g.~\cite[Theorem~X.4.21]{HUL93II}).
Since $\nabla\omega^*(z_k)=x_k$ and $\hat{z}_k-z_k=\alpha_k(z_k-z_{k-1})$,
combining it with the smoothness inequality
$D_{\omega^*}(\hat{z}_k,z_k)=\omega^*(\hat{z}_k)-\omega^*(z_k)-\langle\nabla\omega^*(z_k),\hat{z}_{k}-z_k\rangle\le\frac1{2\mu}\|\hat{z}_k-z_k\|^2$ leads to
\begin{equation*}
\omega^*(\hat{z}_k) \le \omega^*(z_k) + \alpha_k \langle x_k, z_k - z_{k-1} \rangle + \frac{\alpha_k^2}{2\mu}\|z_k - z_{k-1}\|^2.
\end{equation*}
Then, substituting the upper bound of $\omega^*(\hat{z}_k)$ for~\eqref{eq:phi_dual}, $\Phi$ is bounded from below as follows:
\begin{align*}
	\Phi &\ge \alpha_k \big[ \omega^*(z_k) - \omega^*(z_{k-1}) - \langle x_k, z_k - z_{k-1} \rangle \big] - \frac{\alpha_k^2}{2\mu}\|z_k - z_{k-1}\|^2 \\
		&=-\alpha_k D_{\omega^*}(z_{k-1},z_k) - \frac{\alpha_k^2}{2\mu}\|z_k-z_{k-1}\|^2,
\end{align*}
where the last equality holds from $x_k=\nabla\omega^*(z_k)$.
Applying once again the smoothness property $D_{\omega^*}(z_{k-1}, z_k) \le \frac{1}{2\mu}\|z_k - z_{k-1}\|^2$, we obtain
\begin{equation*}
\Phi \ge -\frac{\alpha_k}{2\mu}\|z_k - z_{k-1}\|^2 - \frac{\alpha_k^2}{2\mu}\|z_k - z_{k-1}\|^2 = -\frac{\alpha_k(1+\alpha_k)}{2\mu}\|z_k - z_{k-1}\|^2.
\end{equation*}
Rearranging the terms yields \eqref{eq:dual_inertial_lemma}.
\end{proof}
\begin{remark}
The error term $\delta_k$ of~\eqref{eq:dual_inertial_lemma} depends on the gradient difference $\|\nabla\omega(x_k)-\nabla\omega(x_{k-1})\|$, which measures the inertial step magnitude in the dual space. In the Euclidean case $\omega(x)=\frac12\|x\|^2$, this error term reduces to $\frac{\alpha_k(1+\alpha_k)}{2}\|x_k-x_{k-1}\|^2$ as the dual and primal norms coincide, thereby recovering the classical inertial bound.
\end{remark}

The main results of this section will be established under the following assumptions.
\bassu \lab{assu:var}
The following conditions hold:
\begin{enumerate}[label=(\alph*), ref=(\alph*)]
\item \lab{assu:var.ter}
The sequence $ \seq{\eta_k} $ is nonincreasing, i.e., $ \eta_k \geq \eta_{k+1} $ for all $ k \geq 0 $.
\item \lab{assu:var.terb}
We have $ \alpha_0 \in [0, 1] $ and the sequence $ \seq{\alpha_k \eta_k^{-1}} $ nonincreasing, i.e., 
$ \alpha_k \eta_k^{-1} \geq \alpha_{k+1} \eta_{k+1}^{-1} $ for all $ k \geq 0 $. 
\item \lab{assu:var.seg}
The sequence of stepsizes $ \seq{\lambda_k} $ is chosen to satisfy: $ \lambda_k \in [\,\underline\lambda, \overline\lambda\,] $ for all $ k \geq 0 $,
where $ 0 < \underline\lambda \leq \overline\lambda \leq \mu/L_0 $
and $ L_0 \coloneqq L_F + \eta_0 L_H $.
\item \lab{assu:var.qua} 
The series $ \sum_{k=0}^\infty\,\delta_k \eta_k^{-1} $ is summable, i.e. 
$ s\coloneqq \sum_{k=0}^\infty\, \delta_k \eta_k^{-1} < +\infty $, where $ \delta_k \geq 0 $ is defined as in Lemma~\ref{lem:dual_bregman_inertia}.
\end{enumerate}
\eassu

\brema
\lab{rem:var} \label{rm:var}
By Assumptions~\ref{assu:var}\ref{assu:var.ter} and~\ref{assu:var.terb}, the sequence $(\alpha_k)$ is also nonincreasing, i.e., $\alpha_k\ge\alpha_{k+1}$ for all $k\ge 0$.
From Assumptions~\ref{assu:var}\ref{assu:var.ter} and~\ref{assu:var.seg}, we have $ 0< \lambda_k \leq \mu/L_k $, and hence 
$1- \frac{\lambda_k^2 L_k^2}{\mu^2} \geq 0$ for all $ k \geq 0 $, where $ L_k \coloneqq L_F + \eta_k L_H $.
The analysis is performed in a Euclidean setting, where the strong convexity parameter of $\omega(x) = \frac{1}{2}\|x\|^2$ is $\mu=1$.
In this case, condition~\ref{assu:var.seg} becomes $\lambda_k \le 1/L_k$.
This coincides with the condition imposed in \cite[Assumption 3.1(c)]{alves2025inertial}.
Regarding Assumptions~\ref{assu:var}\ref{assu:var.ter}, note that $x_k$ and $x_{k-1}$ are available before choosing $\alpha_k$, the assumption $\sum_{k=0}^{\infty}\delta_k \eta_k^{-1}<\infty$ can be enforced adaptively at each iteration.
For possible realizations for the sequences $ \seq{\alpha_k } $ and $ \seq{\eta_k } $, see~\cite[Remark 3.2(iii)]{alves2025inertial}. 
Typical choices include 
(i) if  $\eta_k$
is nonincreasing,
the adaptive choice of
$\alpha_k$
proposed in~\cite[Remark~3.2(iii)]{alves2025inertial} guarantees that
$\{\alpha_k/\eta_k\}$
is nonincreasing and
$\sum_k\delta_k/\eta_k<\infty$;
(ii) if
$\eta_k\equiv\eta>0$
is constant,
then it is sufficient to choose
$\{\alpha_k\}$
as any nonincreasing summable sequence.

\erema

To establish the iteration complexity results of Algorithm~\ref{alg:bregman}, we define the ergodic mean $\overline{y}_k$ based on the sequences $\{\lambda_j\}$ and $\{y_j\}$ generated by it:
\begin{align} \label{eq:def.ubs_alt1}
    \overline y_k \coloneqq \frac{1}{\Lambda_k}\sum_{j=0}^{k-1}\lambda_j y_j, \qquad \text{where} \quad
    \Lambda_k \coloneqq \sum_{j=0}^{k-1}\lambda_j.
\end{align}
From Assumption~\ref{assu:var}\ref{assu:var.seg} and the definition of $\Lambda_k$, it follows directly that
\begin{align} \label{eq:play_alt1}
	\Lambda_k \geq \underline\lambda k \qquad \text{for all } k \geq 1.
\end{align}

\begin{theorem}\label{prop result1}\sloppy
Suppose that the sequences generated by \emph{Algorithm~\ref{alg:bregman}} satisfy \emph{Assumption~\ref{assu:var}}. Let $(\overline{y}_k)$ be the sequence of the ergodic mean. Let $\Omega_X^2 \coloneqq \sup_{u,v \in X} D_\omega(u,v)$ be the Bregman diameter of the set $X$. Then, the following statements hold for all $k \ge 1$:
\begin{description}
    \item[\textbf{$\mathrm{(a)}$ Feasibility:}]
    \begin{equation}
    0 \le \operatorname{Gap}(\overline{y}_k, F, X) \le \frac{1}{k\underline{\lambda}}
    \left(
    \Omega_X^2 +
    \sum_{j=0}^{k-1} \delta_j +
    \overline{\lambda} C_H \mathcal{D}_X \sum_{j=0}^{k-1} \eta_j 
    \right)
    \end{equation}
    where $C_H$ is given in~\eqref{eq:def.ch};
    \item[\textbf{$\mathrm{(b)}$ Optimality:}] 
    \begin{equation}
    -B_H \operatorname{dist}(\overline{y}_k, Q) \le
    \operatorname{Gap}(\overline{y}_k, H, Q) \le
    \frac1{k\underline{\lambda}}
    \left(
        \frac{\Omega_X^2}{\eta_k} + \sum_{j=0}^{k-1} \frac{\delta_j}{\eta_j}
    \right)
    \end{equation}
    where $B_H > 0$ is defined as in~\eqref{eq:def.ch} and $\delta_j$ is the inertial error~\eqref{eq:delta}.
    Moreover, if $Q$ is $\sigma$-weakly sharp of order $\mathcal{M}\ge 1$, then for all $k\ge 1$,
    \begin{equation}
        -\frac{B_H}{\sigma^{1/\mathcal{M}}}\left\{
            \frac1{k\underline{\lambda}}\left(
                \Omega_X^2 + \sum_{j=0}^{k-1}\delta_j + \bar{\lambda}C_H\mathcal{D}_X\sum_{j=0}^{k-1} \eta_j
            \right)
        \right\}^{\frac1{\mathcal{M}}}
        \le \Gap(\bar{y}_k,H,Q).
    \end{equation}
\end{description}
\end{theorem}

\begin{proof}
\textbf{$\mathrm{(a)}$ Feasibility:} Let $x\in X$ be arbitrary, $y_j\in X$ and $0\le j \le k-1$.
From Proposition~\ref{prop 3point}, we obtain
\[
    \lambda_j\langle G_j(y_j),y_j-x\rangle=
    \lambda_j \langle F(y_j) + \eta_j H(y_j), y_j - x \rangle \le D_\omega(x, w_j) - D_\omega(x, x_{j+1}).
\]
Here, the last term of~\eqref{ieq:three.point} is omitted due to Remark~\ref{rem:var}.
By the monotonicity of $F$, we have $\langle F(y_j), y_j - x \rangle \ge \langle F(x), y_j - x \rangle$. Rearranging it to isolate the feasibility term, we obtain
\[
    \lambda_j \langle F(x), y_j - x \rangle \le D_\omega(x, w_j) - D_\omega(x, x_{j+1}) - \lambda_j \eta_j \langle H(y_j), y_j - x \rangle.
\]

Now, from the Cauchy-Schwarz inequality, we can bound the term involving~$H$,
i.e., $\langle H(y_j), y_j - x \rangle \le \|H(y_j)\| \|y_j-x\| \le C_H \mathcal{D}_X$, which yields
\[
    \lambda_j \langle F(x), y_j - x \rangle \le D_\omega(x, w_j) - D_\omega(x, x_{j+1}) + \lambda_j \eta_j C_H \mathcal{D}_X.
\]
Let $\varphi_j=D_\omega(x,x_j)$.
Then, it follows from Lemma~\ref{lem:dual_bregman_inertia} that
\[
\begin{aligned}
    \lambda_j \langle F(x), y_j - x \rangle & \le \varphi_j - \varphi_{j+1} + \alpha_j (\varphi_j - \varphi_{j-1}) + \delta_j + \lambda_j \eta_j C_H \mathcal{D}_X \\
\end{aligned}
\]
where $\delta_j$ is the error term given by~\eqref{eq:delta}.
The rest of the proof is based on~\cite[Propositions~3.3~and~3.5]{alves2025inertial};
we omit the detailed algebraic manipulation.
Summing up the above inequalities for $j=1,\dots,k-1$, we then obtain
\[
\begin{aligned}
    \sum_{j=0}^{k-1}\lambda_j\langle F(x),y_j-x\rangle & \le 
    \varphi_{k-1} - \varphi_k + \delta_j + \sum_{j=0}^{k-1} \lambda_j\eta_j C_H \mathcal{D}_X\\
    & \le\Omega_X^2 + \delta_j + \sum_{j=0}^{k-1} \lambda_j\eta_j C_H \mathcal{D}_X.
\end{aligned}
\]
Dividing the above expression by $\Lambda_k$, using~\eqref{eq:play_alt1}, and taking the supremum over $x \in X$ gives the desired bound for $\operatorname{Gap}(\overline{y}_k, F, X)$.

\textbf{$\mathrm{(b)}$ Optimality:} Let $x\in Q$.
We start again from Proposition~\ref{prop 3point}.
\begin{align}\label{ieq:monotone.opt}
    \lambda_j \langle F(y_j) + \eta_j H(y_j), y_j - x \rangle \le D_\omega(x,w_j) - D_\omega(x,x_{j+1}).
\end{align}
By the monotonicity of $F$ and $H$, and since $x\in Q$ implies $\langle F(x),y_j-x\rangle \ge 0$, we have
\begin{align*}
    \langle F(y_j), y_j - x \rangle &\ge \langle F(x), y_j - x \rangle \ge 0, \\
    \langle H(y_j), y_j - x \rangle &\ge \langle H(x), y_j - x \rangle.
\end{align*}
Substituting these into~\eqref{ieq:monotone.opt} yields
\[
    \lambda_j \eta_j \langle H(x), y_j - x \rangle \le D_\omega(x, w_j) - D_\omega(x, x_{j+1}).
\]
Let $\varphi_j=D_\omega(x,x_j)$.
Then, it follows from Lemma~\ref{lem:dual_bregman_inertia} that
\[
    \lambda_j \eta_j \langle H(x), y_j - x \rangle \le \varphi_j - \varphi_{j+1} + \alpha_j(\varphi_j - \varphi_{j-1}) + \delta_j.
\]
Dividing both sides of the inequality by $\eta_j$ and performing some algebraic manipulations, we obtain the following telescoping sum:
\begin{align*}
    \lambda_j \langle H(x), y_j - x \rangle 
    &\le \frac{\varphi_j}{\eta_j} - \frac{\varphi_{j+1}}{\eta_j} + \frac{\alpha_j}{\eta_j}(\varphi_j - \varphi_{j-1}) + \frac{\delta_j}{\eta_j} \\
    &= \left(\frac{\varphi_j}{\eta_j} - \frac{\varphi_{j+1}}{\eta_{j+1}}\right) + \left(\frac{1}{\eta_{j+1}} - \frac{1}{\eta_j}\right)\varphi_{j+1} \\
    &\quad + \left(\frac{\alpha_j \varphi_j}{\eta_j} - \frac{\alpha_{j-1}\varphi_{j-1}}{\eta_{j-1}}\right) + \left(\frac{\alpha_{j-1}}{\eta_{j-1}} - \frac{\alpha_j}{\eta_j}\right)\varphi_{j-1} + \frac{\delta_j}{\eta_j}.
\end{align*}
The rest of the proof also holds similarly to~\cite[Proposition~3.3]{alves2025inertial}.
Summing up this inequality from $j=0$ to $k-1$, 
and using the definition of $\Omega_X^2$ as well as the fact that $\varphi_j \ge 0$, we have
\[ 
  \sum_{j=0}^{k-1} \lambda_j \langle H(x), y_j - x \rangle 
  \le \frac{\Omega_X^2}{\eta_k} + \sum_{j=0}^{k-1} \frac{\delta_j}{\eta_j}.
\]

The latter part under weak sharpness directly follows from (a) and Proposition~\ref{pro:optm02}.
\end{proof}

\subsection{Iteration complexity of Algorithm \ref{alg:bregman} with diminishing and constant regularization parameters} \lab{subsec:dimish}

In the rest of this section, we demonstrate the optimality and feasibility iteration complexity of Algorithm~\ref{alg:bregman} under which the regularization parameter $\eta_k$ is either diminishing or constant.

First, we consider the case in which $\eta_k$ is decreasing.

\bassu \lab{assu:dimish}
The sequence of regularization parameters $ \seq{\eta_k} $ is given by
\begin{align}\lab{eq:def.etak}
 \eta_k \coloneqq \dfrac{\eta_0}{(k+1)^b} \qquad \mbox{for all} \quad k \geq 0,
\end{align}
where $ 0 < \eta_0 \leq 1 $ and $ 0 < b< 1$. 
\eassu 

Now we give a specific iteration complexity of Algorithm~\ref{alg:bregman} with diminishing
regularization.
\bcoro \lab{the:main}
Suppose \emph{Assumptions} \emph{\ref{assu:var}\ref{assu:var.terb}}--\emph{\ref{assu:var}\ref{assu:var.qua}} and \emph{\ref{assu:dimish}} hold.
Let $ \seq{\overline y_k} $ be as in 
\eqref{eq:def.ubs_alt1}
and let the constants $ C_H $ and $ B_H $ be as in \eqref{eq:def.ch}. 
Then the following statements hold:
\begin{description}
    \item [\textbf{$\mathrm{(a)}$ Feasibility:}] \lab{theo:main.ter}
    For all $ k \geq 1 $,
    \begin{align} \lab{eq:sadnight02} 
        0
    	\leq \Gap(\overline y_k, F, X) 
     	\leq \dfrac{1}{k}\left(\dfrac{\Omega_X^2}{\underline\lambda}\right)
     	+ \dfrac{1}{k}\left(\dfrac{s}{\underline \lambda}\right)
     		+ \dfrac{1}{k^b}\left(\dfrac{ \eta_0 \overline \lambda C_H D_X}{(1 - b)\underline \lambda}\right).
    \end{align}
    \item [\textbf{$\mathrm{(b)}$ Optimality:}] \lab{theo:main.seg}
    For all $ k \geq 1 $,
    \begin{align} \lab{eq:sadnight}
        -B_H \dist(\overline y_k, Q) 
        \leq \Gap(\overline y_k, H, Q)
        \leq 
        \dfrac{1}{k^{1-b}}\left(\dfrac{\Omega_X^2 }{2^{-b}\underline \lambda\, \eta_0}\right) + 
        \dfrac{1}{k}\left(\dfrac{s}{\underline \lambda}\right).
    \end{align}
    Moreover, if $Q$ is $\sigma$-weakly sharp of order $\mathcal{M}\ge 1$, then for all $k\ge 1$,
    \begin{align} \lab{eq:sadnight03}
     \Gap(\overline y_k, H, Q) 
     	\geq -\dfrac{B_H}{\sigma^{\,1/\mathcal{M}}} 
     \Bigg( \dfrac{1}{k}\left(\dfrac{\Omega_X^2}{\underline\lambda}\right)
     	+ \dfrac{1}{k}\left(\dfrac{s}{\underline \lambda}\right) 
     		+ \dfrac{1}{k^b}\left(\dfrac{ \eta_0 \overline \lambda C_H D_X}{(1 - b)\underline \lambda}\right)
     \Bigg)^
     {\frac{1}{\mathcal{M}}}.
    \end{align}
\end{description}

\ecoro

\bproo
Since $ \seq{\eta_k} $ as in \eqref{eq:def.etak} clearly satisfies Assumption \ref{assu:var}\ref{assu:var.ter}, then all contents of Assumption \ref{assu:var} hold.
Consequently, the conditions of Theorem~\ref{prop result1} are satisfied.

\textbf{Proof of (a):}
By \eqref{eq:def.etak}, we observe that
\begin{align} \label{eq:integral}
\sum_{j=0}^{k-1}\eta_j = \eta_0\sum_{j=1}^k \frac{1}{j^b} \leq \eta_0 \left(1 + \int_{1}^k \frac{1}{\tau^b} d\tau \right) \leq \frac{\eta_0}{1-b} k^{1-b}.
\end{align}
Then, \eqref{eq:sadnight02} is obtained by combining Theorem~\ref{prop result1} with \eqref{eq:integral} and the fact that $\sum_{j=0}^\infty \delta_j \leq s$ follows from \eqref{eq:def.etak}.

\textbf{Proof of (b):}
The inequality \eqref{eq:sadnight} follows directly from Theorem~\ref{prop result1}, the definition of $s$ in Assumption~\ref{assu:var}\ref{assu:var.qua}, and the bound
\[
k\eta_k = \eta_0 k(k+1)^{-b} \geq \eta_0 2^{-b} k^{1-b}.
\]
The latter part is obvious from Proposition~\ref{pro:optm02} and~\eqref{eq:sadnight02}. 
\eproo

\mgap

We now consider the case in which $\eta_k$ is constant for all $k$.

\bassu \lab{assu:etacons}
The sequence of regularization parameters $ \seq{\eta_k} $ is constant:
\begin{align*} 
 \eta_k \equiv \eta > 0 \qquad \mbox{for all} \quad k \geq 0.
\end{align*}
\eassu

\brema \lab{rem:block}
We notice that Assumption~\ref{assu:var}\ref{assu:var.qua} is reduced to $\sum_{k=0}^\infty \delta_k $ being summable, i.e.,
\begin{align} \lab{eq:def.heta}
 \what s \coloneqq \sum_{k=0}^\infty \delta_k < + \infty.
\end{align}
\erema

\bcoro \lab{the:ofcons}
Suppose that \emph{Assumptions}~\emph{\ref{assu:var}\ref{assu:var.terb}}--\emph{\ref{assu:var}\ref{assu:var.qua}}, \emph{\ref{assu:etacons}} hold\,---\,see Remark \ref{rem:block}.
Let $ \seq{\overline y_k} $, $ \what s $, $ B_H $, and $ C_H $ be as defined in~\eqref{eq:def.ubs_alt1}, \eqref{eq:def.heta}, and~\eqref{eq:def.ch}, respectively.
Then, the following statements hold: 
\begin{description}
    \item[\textbf{$\mathrm{(a)}$ Feasibility:}] 
    For all $ k \geq 1 $,
    \begin{align} \lab{eq:sadday02}
        0 
    	\leq \Gap(\overline y_k, F, X)
     	\leq \dfrac{1}{k}\left( \dfrac{\Omega_X^2 + \what s}{\underline\lambda} \right)
    		+ \eta \left( \dfrac{\overline\lambda C_H D_X}{\underline\lambda}\right).
    \end{align}
    \item[\textbf{$\mathrm{(b)}$ Optimality:}] 
    For all $ k \geq 1 $,
    \begin{align} \lab{eq:sadday}
    -B_H \dist(\overline y_k, Q) 
    	\leq \Gap(\overline y_k, H, Q)
     	\leq 
     	\dfrac{1}{k}\left( \dfrac{\Omega_X^2 + \what s}{\underline\lambda \eta} \right).
    \end{align}
    Moreover, if $Q$ is $\sigma$-weakly sharp of order $\mathcal{M}\ge 1$, then for all $k\ge 1$,
    \begin{align} \lab{eq:sadday03}
     \Gap(\overline y_k, H, Q) 
     	\geq -\dfrac{B_H}{\sigma^{\,1/\mathcal{M}}} 
     		\Bigg( \dfrac{1}{k}\left( \dfrac{\Omega_X^2 + \what s}{\underline\lambda} \right)
    		+ \eta \left( \dfrac{\overline\lambda C_H D_X}{\underline\lambda}\right) \Bigg)^{\frac{1}{\mathcal{M}}}.
    \end{align}
\end{description}
\ecoro
\begin{proof}
    We can show them in the same manner as the proof of Corollary~\ref{the:main}.
\end{proof}

By setting the tolerance $\epsilon>0$ in the general bounds in Corollary~\ref{the:ofcons}, we obtain the following iteration complexity result.

\bcoro \lab{the:ofcons02}
Suppose that all assumptions of \emph{Corollary \ref{the:ofcons}} hold and that
\begin{align} \lab{eq:etaconsb}
 \eta \coloneqq \dfrac{\veps}{2\mathcal{D}_0},
\end{align}
where $ \veps >0 $ is a given tolerance and $ \mathcal{D}_0 > 0 $ is such that
\begin{align} \lab{eq:ubdz}
\mathcal{D}_0 \geq \dfrac{\overline \lambda C_H D_X}{\underline \lambda}.
\end{align}
Then, for all $k$ such that
\begin{align} \lab{eq:lagrange}
k \geq 
\max\left\{ 
\left\lceil \dfrac{2}{\veps^2}\left( \dfrac{\mathcal{D}_0(\Omega_X^2 + \what s)}{\underline\lambda}\right) \right \rceil,
\left\lceil \dfrac{2}{\veps}\left( \dfrac{\Omega_X^2 + \what s}{\underline\lambda}\right)\right\rceil
\right\},
\end{align}
we have
\begin{align} \lab{eq:bertsekas}
-B_H \dist(\overline y_k, Q) \leq  
\Gap(\overline y_k, H, Q) \leq \veps \quad \mbox{and} \quad 
0\leq \Gap(\overline y_k, F, X) \leq \veps.
\end{align}
Moreover, if $ Q $ is $ \sigma $-weakly sharp of order $ \mathcal{M} \geq 1 $, then 
\begin{align} \lab{eq:bertsekas02}
\Gap(\overline y_k, H, Q) \geq -\left(\dfrac{B_H}{\sigma^{1/\mathcal{M}}}\right)
 \veps^{\frac{1}{\mathcal{M}}}.
\end{align}
\ecoro
\bproo
The proof follows directly from Corollary~\ref{the:ofcons}, \eqref{eq:etaconsb} and \eqref{eq:ubdz}.
\eproo

\section{Numerical experiments}\label{sec:ne}

We conduct numerical experiments with Python 3.12 to check the validity of our proposed method. First, we instantiate our framework under two distinct Bregman divergences, providing exact closed-form updates for both the proximal mappings and the dual inertial extrapolations.
\begin{enumerate}[label={}]
    \item \textbf{Euclidean distance} (Euc): $\omega(x) = \frac{1}{2}\|x\|^2$.
    \begin{itemize}
        \item \textit{Dual extrapolation:} The gradient and that
        of the conjugate function are identity mappings, i.e., $\nabla\omega(x) = x$ and $\nabla\omega^*(z) = z$, respectively.
        Consequently, the dual extrapolation degenerates into the standard linear extrapolation in the primal space: $w_k = x_k + \alpha_k(x_k - x_{k-1})$.
        \item \textit{Proximal step:} The prox-mapping is reduced to the standard Euclidean orthogonal projection $P_X$ onto $X$.
    \end{itemize}
    \item \textbf{Kullback-Leibler divergence} (KL): $\omega(x) = \sum x_i \ln x_i$.
    \begin{itemize}
        \item \textit{Dual extrapolation:} The gradient and that of conjugate are given by $[\nabla\omega(x)]_i = 1 + \ln x_i$ and  $[\nabla\omega^*(z)]_i = \exp(z_i - 1)$, respectively. Applying the dual extrapolation step $w_k = \nabla\omega^* \big( (1+\alpha_k)\nabla\omega(x_k) - \alpha_k \nabla\omega(x_{k-1}) \big)$ leads to the \textit{multiplicative geometric extrapolation} as follows:
        \begin{equation}\label{eq:kl_extrapolation}
            [w_k]_i = x_{k, i} \left( \frac{x_{k, i}}{x_{k-1, i}} \right)^{\alpha_k}.
        \end{equation}
        This update ensures that $w_k>0$, thereby strictly preserving interior feasibility and avoiding the boundary violations typically induced by Euclidean linear extrapolation.
        \item \textit{Proximal step:} As detailed in \cite{Beck2017}, the proximal step is calculated via an analytic exponential update. This enforces non-negativity and avoids computationally expensive projections; see~\eqref{prox.explicit} in Section~\ref{ssec:NE.selection} for more details.
    \end{itemize}
\end{enumerate}

In all numerical experiments, the inertial parameter $\alpha_k$ is dynamically updated as $\alpha_k = 0.5 / \sqrt{k+1}$. 
Also, to validate the regularized extragradient scheme, we compare the two strategies for the regularization parameter $\eta_k$:
\begin{itemize}
    \item \textbf{Dynamic $\eta$} (Dyn $\eta$): The parameter monotonically decreases as $\eta_k = 0.1 / (k+1)^{0.5}$.
    \item \textbf{Fixed $\eta$} (Fix $\eta$): The parameter is constant at $\eta_k \equiv 0.1$.
\end{itemize}

\subsection{Equilibrium selection for Nash equilibrium problems}\label{ssec:NE.selection}

In this experiment, we consider the equilibrium selection problem arising from the multi-portfolio problem.
Recently,~\cite{Lampariello2021} proposed a (multi-)~portfolio model in which
several investors make decisions about the portfolio while considering the actions of
other investors, or the portfolio, noncooperatively.

Let $N$ be the number of accounts, which are indexed by $\nu\in\{1,\dots,N\}$. Then, the notations are defined below. For a more detailed description regarding this model, see~\cite{Lampariello2021}:
\begin{itemize}
    \item $K\in\N$: the number of assets to invest;
    \item $b^\nu\in\R_+$: each account $\nu$'s budget to be invested in $K$ assets of a market;
    \item $c^\nu\colon\R^{NK} \to\R^K$: the market impact unitary cost function. We define
    \[
        c^\nu(x^\nu,x^{-\nu}) \coloneqq \Omega^\nu\sum_{\nu'=1}^N b^{\nu'}(x^{\nu'}-v^{\nu'}),
    \]
    where $\Omega\in\R^{K\times K}$ is a market impact matrix whose $(i,j)$ entry 
    is the impact of the liquidity of asset $i$ on the liquidity of asset $j$ and 
    positive semidefinite. The vector $v^\nu\in\R^K$ denotes the current position of account $\nu$;
    \item $r\in\R^K$: the random variable where $r_k$ is the return of asset $k\in\{1,\dots,K\}$ over a single-period investment;
    \item $\mu^\nu\in\R^K$: the expectations of the assets' returns for $\nu$;
    \item $\Sigma^\nu \coloneqq \mathbb E^\nu[(r-\mu^\nu)(r-\mu^\nu)^\top]\in\mathbb S^K_+$: 
    the covariance matrix, where $\mathbb S^K_+$ is a space of symmetric positive semidefinite matrices of $K\times K$;
    \item $x^\nu\in\R^K$: (variable) the fractions of $b^\nu$ to invest in each asset;
    \item $x^{-\nu}\in\R^{(N-1)K}$: tuple of all accounts' variables, except $x^\nu$.
\end{itemize}

Let $I^\nu(x^\nu) \coloneqq b^\nu (\mu^\nu)^\top x^\nu$,
$R^\nu(x^\nu) \coloneqq \frac12(b^\nu)^2 (x^\nu)^\top\Sigma^\nu x^\nu$,
and $TC^\nu(x^\nu,x^{-\nu}) \coloneqq b^\nu(x^\nu-v^\nu)^\top c^\nu(x^\nu,x^{-\nu})$
respectively be a portfolio income, risk, and total transaction cost.
The cost function for account $\nu$ is defined by
\[
    \theta^\nu(x^\nu,x^{-\nu}) \coloneqq -I^\nu(x^\nu)+\rho^\nu R^\nu(x^\nu)+ TC^\nu(x^\nu,x^{-\nu}),
\]
and its strategy space $X^\nu$ is given by a polyhedron as follows:
\begin{align}\label{eq:polyhedron}
    X^\nu \coloneqq \{x^\nu\in\R^K\mid e^\top x^\nu = 1,\ x^\nu \ge 0\},
\end{align}
where $e$ denotes the all-ones vector of appropriate dimension.
The account $\nu\in\{1,\dots,N\}$ solves the following problem:
\begin{align}\label{prob:gnep.nu}
    \min_{x^\nu\in\R^{K}} \quad \theta^\nu(x^\nu,x^{-\nu})
    \qquad \text{s.t.} \quad x^\nu\in X^\nu.
\end{align}

Let $\mathrm{NE}$ be the set of Nash equilibria; that is,
\[
    \mathrm{NE} \coloneqq \left\{
        x\in X\ \middle|\ \theta^\nu(x^{\nu},x^{-\nu})\le \min_{y^\nu\in X^\nu}\theta^\nu(y^\nu,x^{-\nu})\quad \forall \nu\in\{1,\dots,N\}
    \right\},
\]
where 
\[
    X \coloneqq X^1\times \dots \times X^N=\{x\in\R^{NK}\mid e^\top x^\nu = 1,\ x^\nu\ge 0,  \quad \nu = 1,\dots,N\}.
\]

Since each player's objective function is convex, and the strategy set is convex and compact,
the Nash equilibrium of the game exists, i.e., $\mathrm{NE} \neq \emptyset$. 
In this setting, it is known that the game is reformulated as finding a solution to variational inequality~\eqref{eq:def.Q} as follows (see~\cite{Facchinei2004}):
\[
\mathrm{NE}=\mathrm{SOL}(F,X) \coloneqq \{ x\in X \mid \langle F(x), y-x\rangle \ge 0\quad \forall y \in X \},
\]
where $F\colon\R^{NK}\to\R^{NK}$ is defined as 
\[
    F(x) \coloneqq
    \begin{bmatrix}
        \nabla_{x^1} \theta^1(x^1,x^{-1}) \\
        \vdots \\
        \nabla_{x^N} \theta^N(x^N,x^{-N})
    \end{bmatrix}.
\]

Now we consider the following equilibrium selection problem according to some further criterion $f\colon\R^{NK}\to\R$, which is formulated as an optimization problem with variational inequality constraints (OPVIC); see~\cite{Lampariello2021} for details:
\begin{align*}
    \min \quad f(x) \qquad \text{s.t.} \quad x\in\mathrm{SOL}(F,X).
\end{align*}

In this experiment, $f\colon\R^{NK}\to\R$ is supposed to be the following metric of the sparsity function:
\[
    f(x) \coloneqq z^\top \sum_{\nu=1}^N x^\nu.
\]
Here, $z\in\R^K$ is a randomly generated vector of positive weights, and 
we add to the function $f(x)$ a negligible regularization term $\tau x^\top x$ 
so that the resulting sum is strictly convex with small $\tau>0$.
Since $x^\nu\ge 0$, $f(x)$ corresponds to the weighted $\ell_1$-norm of $\sum_{\nu=1}^N x^\nu$,
where $z_i$ denotes the weight associated with the asset $i$.

It should be noted that the proximal operator $\prox^\omega_{\lambda\xi}(x)$
for $\xi:=(\xi^1,\dots,\xi^N)\in\R^{NK}$, $\lambda>0$, and $\omega(x)=\sum_{\nu=1}^N \sum_{i=1}^{K} x^\nu_i \ln x^\nu_i$
is obtained explicitly if we assume $\prox^\omega_{\lambda\xi}(x)>0$~(\cite{Beck2017}):
\begin{align}\label{prox.explicit}
    [\prox^\omega_{\lambda\xi}(x)]^\nu_i=
    \frac{x^\nu_i \exp(-\lambda\xi^\nu_i)}{\sum_{j=1}^K x^\nu_j\exp(-\lambda\xi^\nu_j)},\quad
    i=1,\dots,K,\ \nu=1,\dots,N.
\end{align}
Therefore, we do not need to compute a projection. It is
explicitly obtained, which may contribute to reducing the computation time.

\subsubsection{Numerical data and configurations}\label{sssec:data.generation}
The problem parameters for the $N$-player portfolio game were generated as follows. We consider a scenario with $N=5$ accounts and $K=10$ assets, resulting in a joint strategy space of dimension $n = 50$. 
The parameters are randomly instantiated using a fixed seed:

\begin{itemize}
    \item The budget of each account: $b^\nu \sim U(1.0, 5.0)$, where $U(a,b)$ represents a uniform distribution in the interval $[a,b]$;
    \item The expected return: $\mu^\nu_i \sim U(0.05, 0.20)$;
    \item The covariance matrix: $\Sigma^\nu = \frac{1}{K} A^\top A $, where the entries of $A \in \mathbb{R}^{K \times K}$ are drawn from a standard normal distribution $\mathcal{N}(0,1)$, and its first two columns are set equal;
    \item The market impact: $\Omega = \frac{1}{K} B^\top B$, where $B \in \mathbb{R}^{K \times K}$ has entries from $\mathcal{N}(0,1)$, and its first two columns are set equal;
    \item The reference volume $v^\nu$ is drawn from a flat Dirichlet distribution;
    \item The risk aversion factor: $\rho^\nu \sim U(0.5, 2.0)$;
    \item The weight for the upper-level sparsity metric $f(x)$: $z_i \sim U(0.1, 1.0)$ and the strict convexification parameter $\tau = 10^{-4}$.
\end{itemize}

The construction of $A$ and $B$ makes the first two assets indistinguishable at the lower level.
This yields multiple lower-level equilibria in the NEP.
The common extragradient step-size is fixed at $\lambda_k \equiv 0.005$. 

\subsubsection{Evaluation metrics:}
\begin{enumerate}
    \item \textbf{Suboptimality (stationarity):} $D_k = \|\bar{y}_k - \bar{y}_{k-1}\|$.
    \item \textbf{Lower-level feasibility (NCP residual)}: 
    We track convergence to the lower-level solution set $Q$ via the natural residual:
    \begin{equation} \label{eq:infeasib}
        \phi(\bar{y}_{k}) \coloneqq \| \bar{y}_{k} - P_X(\bar{y}_{k} - F(\bar{y}_{k})) \|^2.
    \end{equation}
    In fact, note that
    $\phi(\bar{y}_k)=0$ is equivalent to $\bar{y}_k \in Q$.
\end{enumerate}

\subsubsection{Results and discussions}

\begin{figure}[ht]
    \centering
    \includegraphics[width=0.8\textwidth]{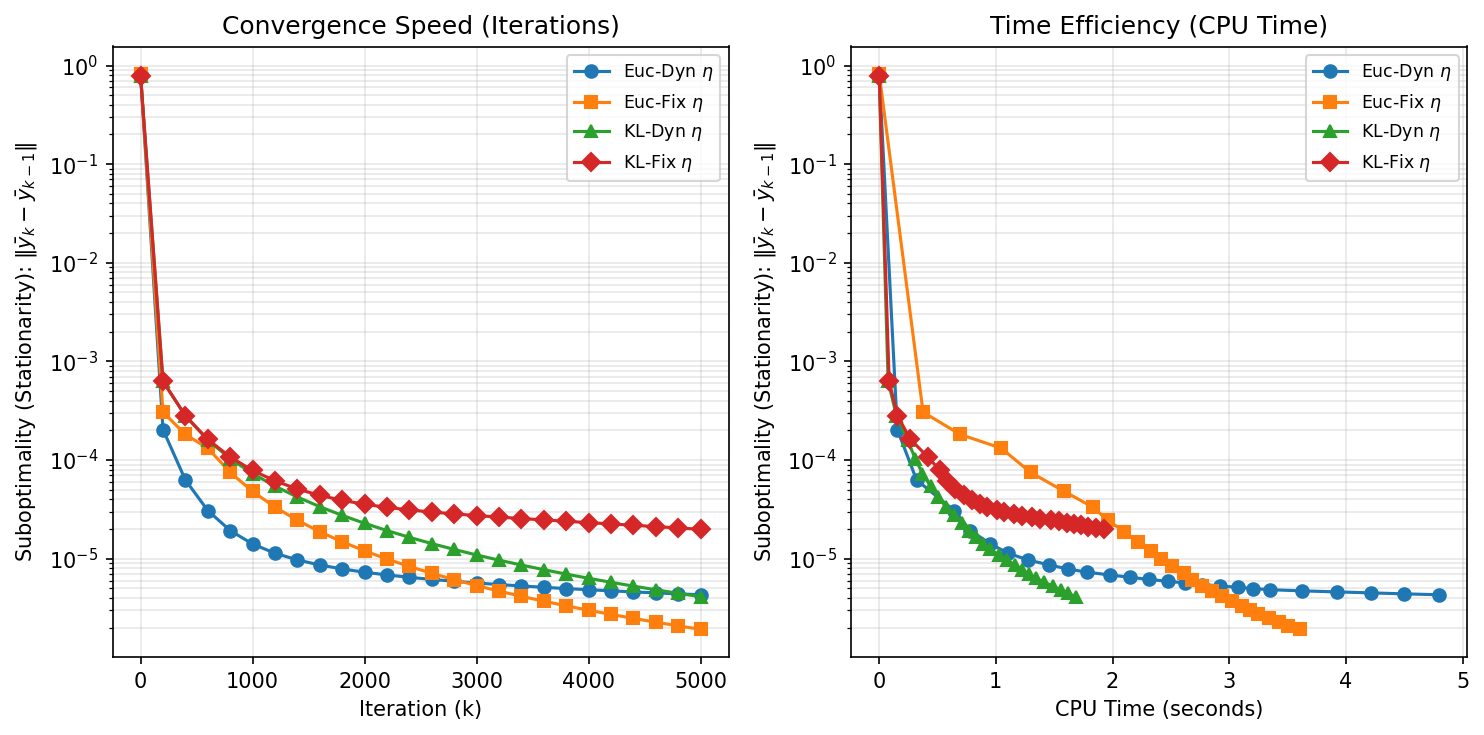}
    \caption{Performance comparison of the \textbf{IneIREG-Bregman} algorithm for the $N$-player portfolio game in terms of suboptimality (stationarity of the ergodic mean $\|\bar{y}_k - \bar{y}_{k-1}\|$). }
    \label{fig:portfolio_subopt}
\end{figure}

\begin{figure}[ht]
    \centering
    \includegraphics[width=0.8\textwidth]{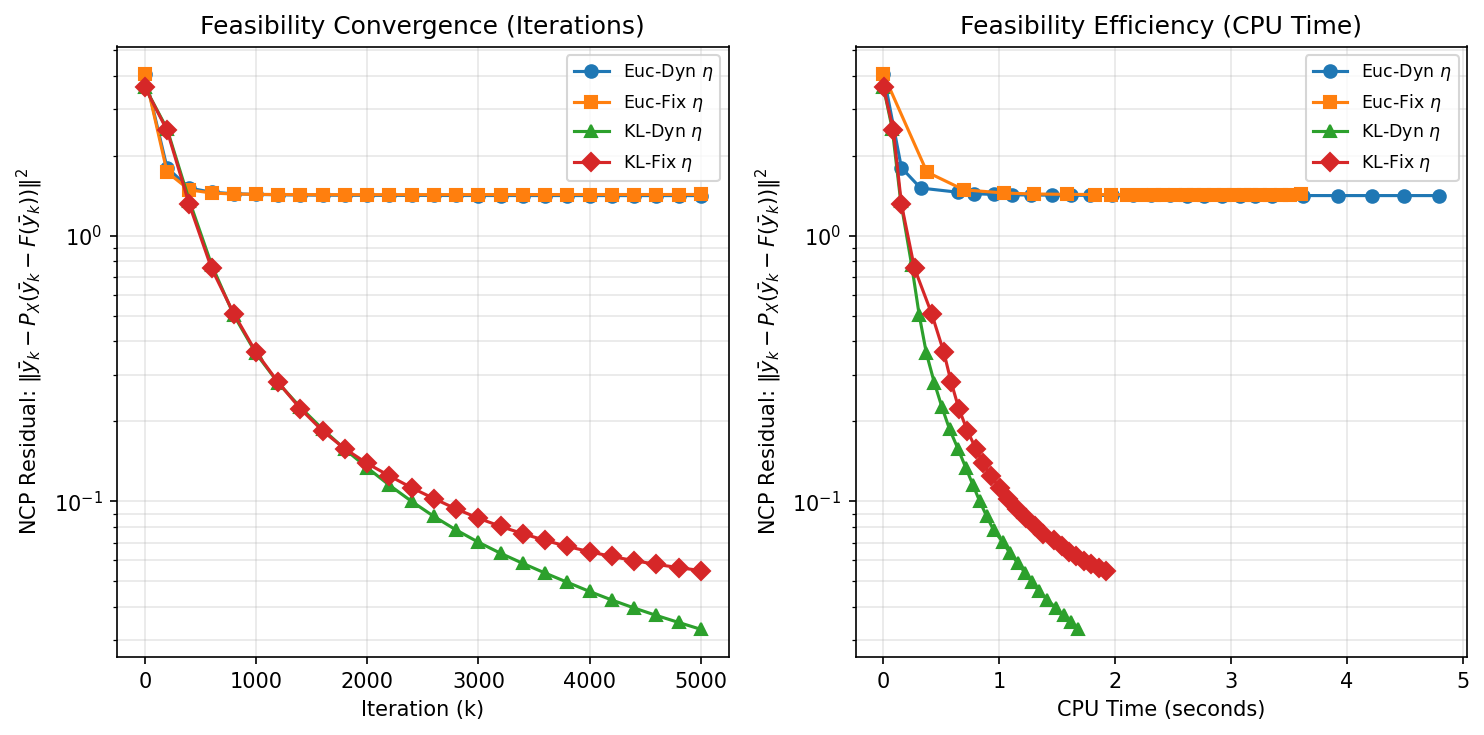}
    \caption{Convergence of $\phi(\bar{y}_k) = \|\bar{y}_k - P_X(\bar{y}_k - F(\bar{y}_k))\|^2$.}
    \label{fig:portfolio_infeasibility}
\end{figure}
\begin{figure}[ht]
    \centering
    \includegraphics[width=0.8\textwidth]{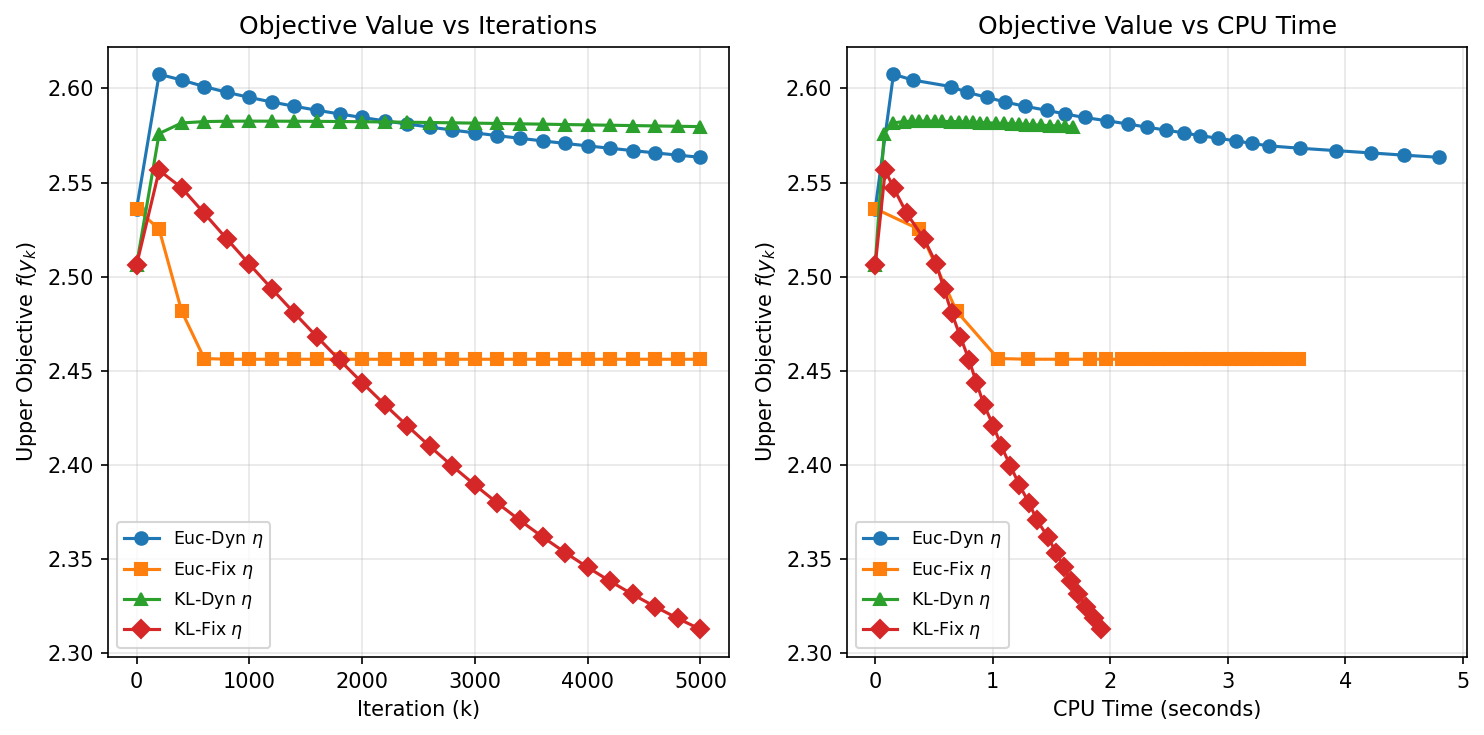}
    \caption{Upper-level objective value}
    \label{fig:nash_value}
\end{figure}
Although the Euclidean variants exhibit small suboptimality values
in Figure~\ref{fig:portfolio_subopt}, Figure~\ref{fig:portfolio_infeasibility}
shows that their natural residuals~\eqref{eq:infeasib} stay at relatively large values.
In contrast, the dual-inertial KL method reduces the equilibrium violation toward zero.
In terms of computational time, the right graph of Figure~\ref{fig:portfolio_infeasibility} shows another advantage. While Euclidean methods rely on $\mathcal{O}(K\log K)$ sorting-based projections per block, the KL framework utilizes explicit $\mathcal{O}(K)$ normalization. It enables the KL variants to compute higher-accuracy equilibria in less time, which demonstrates their efficiency for simplex-constrained games. In addition, Figure~\ref{fig:nash_value} shows that both KL variants attain lower upper-level objective values than their Euclidean counterparts.

\subsection{A traffic equilibrium problem}\label{ssec:traffic}

Similarly to~\cite{alves2025inertial,SamYou25}, we now consider the Nguyen and Dupuis traffic network equilibrium problem. The network, illustrated in Figure~\ref{fig:network}, consists of 13 vertices, 19 arcs, 25 paths, and 4 origin-destination (OD) pairs. Following the standard path-based formulation, the user equilibrium condition is modeled as a Nonlinear Complementarity Problem (NCP).
\begin{figure}[htbp]
    \centering 
    \includegraphics[width=0.6\textwidth]{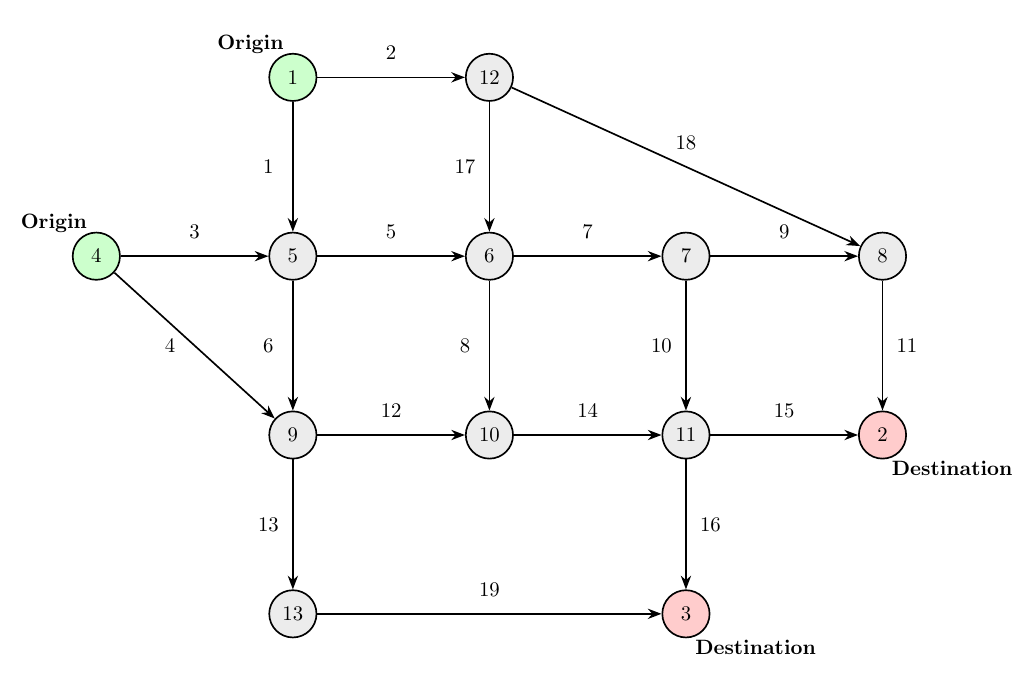} 
    \caption{Nguyen and Dupuis network} 
    \label{fig:network} 
\end{figure}
The goal is to find a state vector $x = [h; u]^\top \in \mathbb{R}^{29}$, where $h \in \mathbb{R}^{25}$ is the vector of traffic flows on the paths, and $u \in \mathbb{R}^{4}$ is the vector of minimum travel costs between the OD pairs. The vector $x$ must satisfy the NCP conditions as follows:
\begin{align} \label{eq:ncp_traffic}
0 \le x \perp F(x) \ge 0,
\end{align}
where $x \perp F(x)$ means $\inner{x}{F(x)} = 0$, and $F \colon \mathbb{R}^{29} \to \mathbb{R}^{29}$ is defined by
\[
F(x) \coloneqq \begin{bmatrix} C(h) - \Omega^\top u \\ \Omega h - d \end{bmatrix}.
\]
Here, $\Omega$ is the OD-path incidence matrix and $d$ is the given travel demand vector. The path cost vector $C(h)$ is derived from the arc travel times via the relation $C(h) = \Delta^\top c(\Delta h)$, where $\Delta$ is the arc-path incidence matrix.
For each arc $a$, the travel time $c_a(\cdot)$ is given by the Bureau of Public Roads (BPR) function:
\[
c_a(\mathcal{F}_a) \coloneqq t_a^0 \left\{1+ 0.15 \left(\frac{\mathcal{F}_a}{cap_a}\right)^{n_a}\right\},
\]
where $\mathcal{F}_a$ is the traffic flow on  $a$, $t_a^0$ is the free-flow time, $cap_a$ is its capacity, and $n_a \ge 1$. It is shown in~\cite[Lemma~6.1]{SamYou25} that the mapping $F$ is monotone when $n_a \ge 0$ for all~$a$. 

The primary objective of this example is to find an equilibrium state that minimizes the total aggregated travel cost across the entire network. Define
\[
f(x) \coloneqq \sum_{i=1}^{25} [C(h)]_i,
\]
which is convex when $n_a \ge 1$ for every arc~$a$ (see \cite[Lemma~6.2]{SamYou25}). The problem is thus stated as finding the best equilibrium, in the following sense:
\begin{align*}
\min \quad f(x)  \qquad
\text{s.t.} \quad x \in \text{SOL}(F,\mathbb{R}^{29}_+),
\end{align*}
where $\text{SOL}(F,\R^{29}_+)$ denotes the solution set of the NCP defined in~\eqref{eq:ncp_traffic}.

\subsubsection{Evaluation metrics:}
\begin{enumerate}
    \item \textbf{Suboptimality (stationarity)}: $D_k = \|\bar{y}_k - \bar{y}_{k-1}\|$.
    
    \item \textbf{Lower-level feasibility (NCP residual):} A merit function to measure constraint feasibility and complementarity, defined as:
    \begin{equation} \label{eq:infeasibility}
      \phi(\bar{y}_{k}) \coloneqq \| \max\{0, -\bar{y}_{k}\} \|^2 + \| \max\{0, -F(\bar{y}_{k})\} \|^2 + |\bar{y}_{k}^\top F(\bar{y}_{k})|.
    \end{equation}
\end{enumerate}

\subsubsection{Results and discussions}

\begin{figure}[ht]
    \centering
    \includegraphics[width=0.8\textwidth]{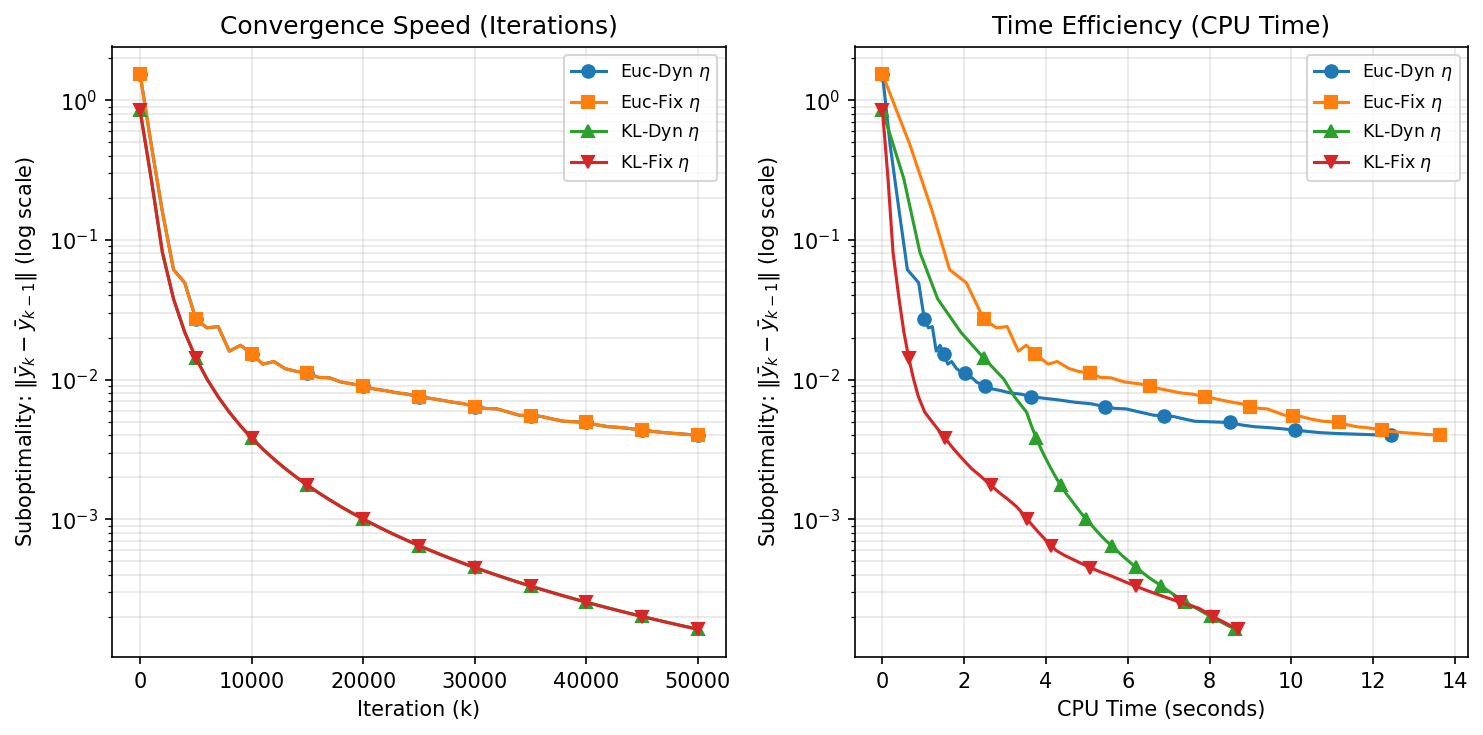}
    \caption{Performance comparison for the traffic network equilibrium problem in terms of suboptimality (stationarity of the ergodic mean $D_k = \|\bar{y}_k - \bar{y}_{k-1}\|$). }
    \label{fig:subopt_iter}
\end{figure}

\begin{figure}[ht]
    \centering
    \includegraphics[width=0.8\textwidth]{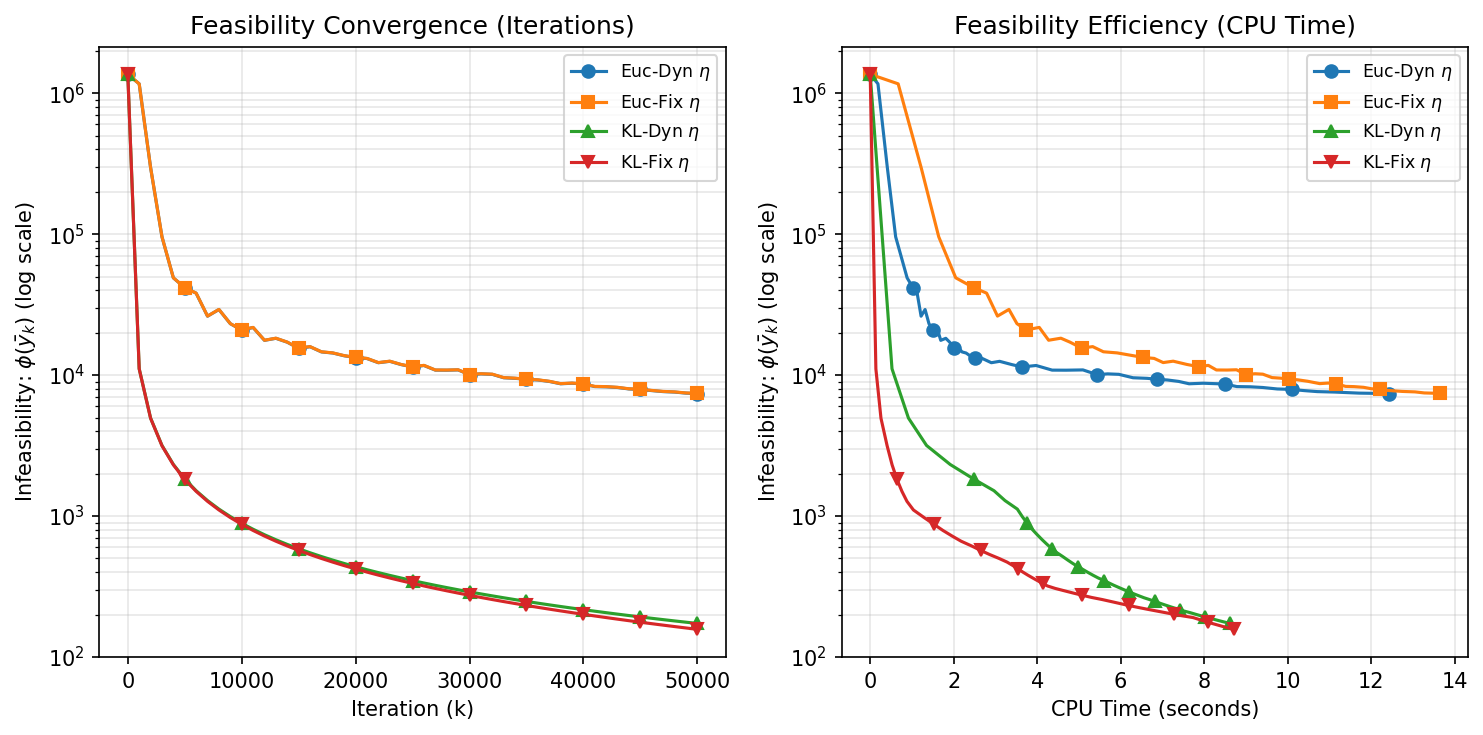}
    \caption{Convergence of the infeasibility $\phi(\bar{y}_k)$ for the traffic network equilibrium problem.  }
    \label{fig:infeas_iter}
\end{figure}

\begin{figure}[ht]
    \centering
    \includegraphics[width=0.8\textwidth]{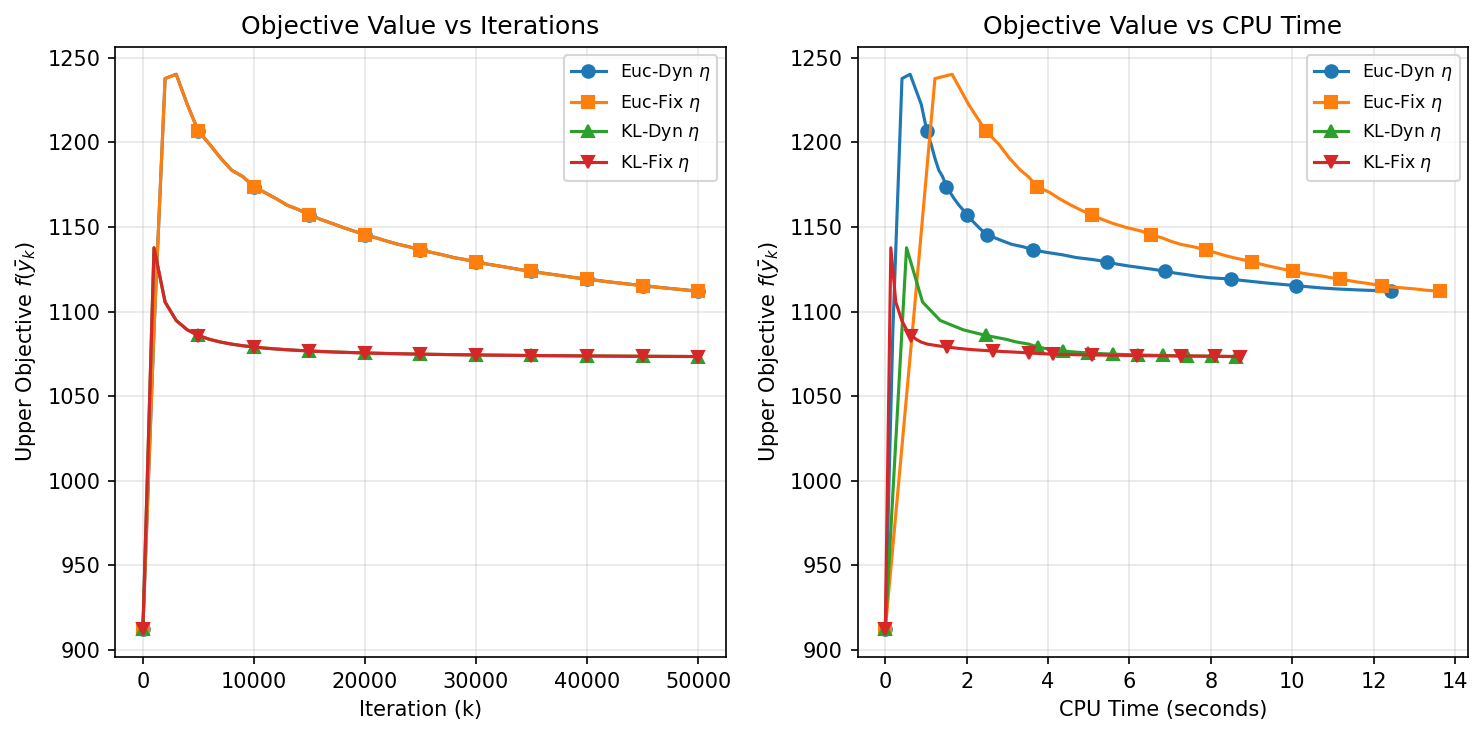}
    \caption{Upper-Level Objective Value}
    \label{fig:value_iter}
\end{figure}

As shown in Figures~\ref{fig:subopt_iter} and~\ref{fig:infeas_iter}, the KL variants exhibit better performance than the Euclidean variants in terms of both suboptimality and infeasibility. In particular, the KL-based methods achieve lower residual values with fewer iterations and less computational time. Moreover, Figure~\ref{fig:value_iter} further confirms that the KL variants attain lower upper-level objective values.

\section{Conclusions}\label{sec:conclusions}

We proposed a Bregman inertial regularized extragradient method for bilevel variational inequalities with inertial extrapolation in the dual space. We show that this framework preserves Bregman three-point identities, which allows the method to provide non-asymptotic complexity bounds under standard monotonicity and Lipschitz assumptions. Our simple numerical experiments show that using Bregman divergence gives computational advantages. In particular, explicit projection-free updates reduce time while guaranteeing interior feasibility. As future work, we plan to extend this dual-inertial framework to settings satisfying relative smoothness conditions~(\cite{LFN18}), which can avoid the global Lipschitz assumption.


\section*{Declarations}


\textbf{Funding} This work was supported by Japan Society for the Promotion of Science with Grant-in-Aid for Early-Career Scientists (JP26K21172), Grants-in-Aid for Scientific Research (C) (JP25K15002) and Grants-in-Aid for Scientific Research (B) (JP25K03082).

\noindent\textbf{Data availability} 
The numerical experiments in Section~\ref{sec:ne} use synthetic data.
The synthetic data-generation procedure and the parameter settings are described in Section~\ref{sssec:data.generation}.
The traffic network model and relevant references are provided in Section~\ref{ssec:traffic}.

\noindent\textbf{Conflict of interest} The authors declare that they have no conflict of interest.

\noindent\textbf{Ethical approval} There are no applicable ethical restrictions.

\noindent\textbf{Informed consent}
Informed consent was not applicable for this study.






\bibliography{reference}

\end{document}